\documentclass[11pt,a4paper]{amsart}
\usepackage{hyperref,  enumitem}
\usepackage{graphicx}
\usepackage{epstopdf}

\newtheorem{theorem}{Theorem}[section]

\newtheorem{corollary}[theorem]{Corollary}
\newtheorem{lemma}[theorem]{Lemma}
\newtheorem{proposition}[theorem]{Proposition}

\theoremstyle{definition}
\newtheorem*{definition}{Definition}
\newtheorem*{claim}{Claim}

\usepackage[square,sort,comma,numbers]{natbib}

\usepackage[all]{xy}
\usepackage{pstricks}
\usepackage{amsfonts,amssymb,amsmath,eucal,pinlabel,array,hhline}
\usepackage{slashed}
\usepackage{tabulary}
\usepackage{fancyhdr}
\usepackage{a4wide}
\usepackage{calrsfs,bbm,dsfont}
\usepackage[position=b]{subcaption}

\newcommand{\Z}{\mathds{Z}}
\newcommand{\Q}{\mathds{Q}}

\newcommand{\eps}{\varepsilon}

\DeclareMathOperator{\lk}{lk}
\DeclareMathOperator{\sgn}{sgn} 
\DeclareMathOperator{\im}{im}
\DeclareMathOperator{\coker}{coker}
\DeclareMathOperator{\Span}{Span}
\DeclareMathOperator{\Hom}{Hom}
\DeclareMathOperator{\ev}{ev}
\DeclareMathOperator{\clasp}{clasp}
\DeclareMathOperator{\interior}{int}

\def\sm{\setminus}

\def\L{\Lambda}

\def\op{\operatorname}

\newcommand{\smoplus}[2]{\mbox{\footnotesize$\displaystyle\bigoplus\limits_{#1}^{#2}$}}

\newcommand{\smprod}[2]{\mbox{\footnotesize$\displaystyle\prod\limits_{#1}^{#2}$}}

\newcommand{\tmfrac}[2]{\mbox{\large$\frac{#1}{#2}$}}

\newcommand{\smsum}[2]{\mbox{\footnotesize$\displaystyle\sum\limits_{#1}^{#2}$}}

\newcommand{\lmcup}[2]{\mbox{\small$\displaystyle\bigcup\limits_{#1}^{#2}$}}
\newcommand{\smcup}[2]{\mbox{\footnotesize$\displaystyle\bigcup\limits_{#1}^{#2}$}} 

\newcommand{\ttmcup}[2]{\mbox{\footnotesize{$\textstyle \bigcup\limits_{#1}^{#2}$}}}

\newcommand{\smsqcup}[2]{\mbox{\footnotesize$\displaystyle\bigsqcup\limits_{#1}^{#2}$}}

\newenvironment{romanlist}
	{\begin{enumerate}
	}
	{\end{enumerate}}

\DeclareSymbolFont{EulerScript}{U}{eus}{m}{n}
\DeclareSymbolFontAlphabet\mathscr{EulerScript}
\begin{document}

\title{The Blanchfield pairing of colored links}
\author{Anthony Conway}
\address{Universit\'e de Gen\`eve, Section de math\'ematiques, 2-4 rue du Li\`evre, 1211 Gen\`eve 4, Switzerland}
\email{anthony.conway@unige.ch}
\author{Stefan Friedl}
\address{Fakult\"at f\"ur Mathematik\\ Universit\"at Regensburg\\   Germany}
\email{sfriedl@gmail.com}
\author{Enrico Toffoli}
\address{Fakult\"at f\"ur Mathematik\\ Universit\"at Regensburg\\   Germany}
\email{enricotoffoli@gmail.com}

\maketitle

\begin{abstract}
It is well known that the Blanchfield pairing of a knot can be expressed using Seifert matrices. In this paper, we compute the Blanchfield pairing of a colored link with non-zero Alexander polynomial. More precisely, we show that the Blanchfield pairing of such a link can be written in terms of generalized Seifert matrices which arise from the use of C-complexes.
\end{abstract}

\section{Introduction}
In the early days of knot theory, most invariants were extracted from the Alexander module $H_1(X_K;\mathbb{Z}[t^{\pm 1}])$ of a knot $K\subset S^3$. (Here, given a knot $K$, we denote by $X_K=S^3\setminus \nu K$ its exterior.) In 1957, Blanchfield~\cite{Bl} showed that the Alexander module supports   a non-singular pairing
$$ \op{Bl}(K)\colon H_1(X_K;\mathbb{Z}[t^{\pm 1}]) \times H_1(X_K;\mathbb{Z}[t^{\pm 1}]) \,\,\rightarrow\,\, \mathbb{Q}(t)/\mathbb{Z}[t^{\pm 1}]$$
which is sesquilinear (i.e.\ linear over~$\mathbb{Z}[t^{\pm 1}]$ in the first variable and conjugate-linear over~$\mathbb{Z}[t^{\pm 1}]$ in the second variable) and hermitian. The Blanchfield pairing can be expressed in terms of a Seifert matrix of $K$, see \cite{Ke, Lev, FP}.

\begin{theorem} \label{thm:blanchfield-in-terms-of-seifert-matrix-for-knot}
If $K$ is a knot and if~$A$ is  a Seifert matrix for $K$ of size~$2g$, then the Blanchfield pairing of~$K$ is isometric to the pairing
\begin{align*}
\mathbb{Z}[t^{\pm 1}]^{2g}/(tA-A^T)\mathbb{Z}[t^{\pm 1}]^{2g} \times  \mathbb{Z}[t^{\pm 1}]^{2g}/(tA-A^T)\mathbb{Z}[t^{\pm 1}]^{2g} &\,\,\rightarrow\,\, \mathbb{Q}(t)/\mathbb{Z}[t^{\pm 1}] \\
(a,b) &\,\,\mapsto\,\, a^T (t-1)(A-tA^T)^{-1}\overline{b}.
\end{align*}
\end{theorem}

Our goal is to extend Theorem~\ref{thm:blanchfield-in-terms-of-seifert-matrix-for-knot} to links.  In this paper, a~$\mu$-{\em colored link\/} is an oriented link~$L$ in~$S^3$ whose components are partitioned into~$\mu$ sublinks~$L_1\cup\dots\cup L_\mu$. Throughout the text, we denote by~$X_L=S^3\sm \nu L$ the exterior of the  link. Furthermore, we write~$\Lambda_S:=\mathbb{Z}[t_1^{\pm1},\dots,t_\mu^{\pm 1},(1-t_1)^{-1},\dots,(1-t_\mu)^{-1}]$ for the localization of the ring of Laurent polynomials, and we denote by~$Q=\mathbb{Q}(t_1,\dots,t_\mu)$  the quotient field of~$\Lambda_S$. Using this setting, the Blanchfield pairing for knots generalizes to a sesquilinear pairing
$$ \op{Bl}(L)\colon TH_1(X_L;\Lambda_S) \times TH_1(X_L;\Lambda_S) \,\,\rightarrow\,\, Q/\Lambda_S$$
on the ~$\Lambda_S$-torsion submodule~$TH_1(X_L;\Lambda_S)$ of the Alexander module $H_1(X_L;\Lambda_S)$  of~$L$. We refer to Section~\ref{sub:Blanchfield} for details. For knots, we recover the Blanchfield pairing from above. (Here we use that in any knot module multiplication by $t-1$ is an isomorphism, i.e.\ the Alexander module over $\Lambda_S=\mathbb{Z}[t^{\pm 1},(1-t)^{-1}]$ is the same as the Alexander module over $\mathbb{Z}[t^{\pm 1}]$, see e.g.\ \cite{Lev} for details.)

In order to extend Theorem~\ref{thm:blanchfield-in-terms-of-seifert-matrix-for-knot} from knots to colored links, we shall make use of C-complexes and generalized Seifert matrices for colored links \citep{Cim,Co,CF}. Roughly speaking, a C-complex for a $\mu$-colored link $L$ consists in a collection of Seifert
surfaces $S_1, \dots , S_\mu$ for the sublinks $L_1, \dots , L_\mu$ that intersect only along clasps. Given such a C-complex and a sequence~$\varepsilon=(\varepsilon_1,\varepsilon_2,\dots, \varepsilon_\mu)$ of $\pm 1$'s, there are $2^\mu$ generalized Seifert matrices~$A^\varepsilon$ which extend the usual Seifert matrix (see Subsection \ref{sub:Seifert} for the details). We set 
$$H(t)\,\,=\,\,\sum_\varepsilon \,\prod_{i=1}^\mu (1-t_i^{\varepsilon_i}) \, A^\varepsilon,$$
where the sum is on all sequences~$\varepsilon=(\varepsilon_1,\varepsilon_2,\dots, \varepsilon_\mu)$ of~$\pm 1$'s. For example, if $L$ is a $1$-colored link (i.e. an oriented link), then 
\[ H(t)\,\,=\,\,(1-t)A^T+(1-t^{-1})A,\]
where $A$ denotes the usual Seifert matrix as defined in \cite{Rolfsen}. Finally, we shall call a C-complex $S$ \textit{totally connected} if each $S_i$ is connected and $S_i \cap S_j  \neq \emptyset $ for all $i\neq j$. 

Our main theorem reads as follows: 

\begin{theorem}
\label{thm:main}
Let $L$ be a $\mu$-colored link. Consider a totally connected C-complex for $L$ and let~$A^\varepsilon$ be the resulting $n \times n$-generalized Seifert matrices. Define $H(t)$ as above. If $H_1(X_L;\Lambda_S)$ is~$\Lambda_S$-torsion, then the Blanchfield pairing $\op{Bl}(L)$ is isometric to the pairing
\begin{align*}
\Lambda_S^n / H(t)^T \Lambda_S^n \times \Lambda_S ^n / H(t)^T \Lambda_S^n &\,\,\rightarrow\,\, Q/\Lambda_S \\
(a,b) &\,\,\mapsto\,\, -a^T H(t)^{-1}\overline{b}.
\end{align*}
$($Here $H(t)^{-1}$ is the inverse of $H(t)$ over $Q$.$)$
\end{theorem}

As the matrix~$H(t)$ is hermitian, Theorem \ref{thm:main} provides a proof that the Blanchfield pairing $\op{Bl}(L)$ is hermitian. To the best of our knowledge, in the case of links, the only other proof of this fact was recently given by Powell~\cite{Powell}. Moreover, we know of no computation of the Blanchfield form for links which are not (homology) boundary links \cite{CochranOrr, Hillman}. 

The main technical ingredient in the proof of Theorem~\ref{thm:main} is the following result, which is also  of independent interest, see e.g.\ \cite{CNT}.
The theorem below is stated with more details in Theorem \ref{thm:intform}.

\begin{theorem}
\label{thm:intform-intro}
Let $L$ be a colored link, let $S$ be a C-complex for $L$ and let $W$ be the exterior of a push-in of the C-complex into the 4-ball $D^4$. 
Let $H(t)$ be the $n\times n$-matrix described above. Then the intersection pairing on $H_2(W;\Lambda_S)$ is represented by the matrix $H(t)$.
\end{theorem}

The statement of Theorem~\ref{thm:intform-intro} is well-known for knots, see e.g.\ \cite[Proof~of~Lemma~5.4]{COT04} and~\cite[Theorem~3]{Ko89}. Note also that a similar statement appears in  \cite[Proof~of~Proposition~1]{Li84}, but there the author obtains the matrix $H(t)^T$ instead of $H(t)$, which we think is slightly incorrect. See also \cite[Proposition 6.5]{CF} for a similar statement involving finite abelian covers branched along a C-complex. Finally, note that contrarily to Theorem \ref{thm:main}, Theorem \ref{thm:intform-intro} requires neither $H_1(X_L;\Lambda_S)$ to be torsion, nor the $C$-complex to be totally connected.

The paper is organized as follows. In Section~$\ref{sec:preliminaries}$, we recall the definitions of generalized Seifert matrices, twisted homology, intersection forms and Blanchfield pairings. In Section \ref{sec: Pushin} and \ref{sec:intform}, we consider the exterior of the ``push-in of a C-complex" in the 4-ball and its twisted intersection form. Finally in Section \ref{sec:proof}, we prove Theorem \ref{thm:main} and show that for knots it recovers Theorem~\ref{thm:blanchfield-in-terms-of-seifert-matrix-for-knot}.

\subsection*{Notation and conventions.}
We denote by $p\mapsto \overline{p}$ the usual involution on $\mathbb{Q}(t_1,\dots,t_\mu)$ induced by $\overline{t_i}=t_i^{-1}$.
Furthermore, given a subring $R$ of $\mathbb{Q}(t_1,\dots,t_\mu)$ closed under the involution, and given an $R$-module $M$, we denote by $\overline{M}$ the $R$-module that has the same underlying additive group as $M$, but for which the action by $R$ on $M$ is precomposed with the involution on $R$. Finally, given any ring $R$,  we think of elements in $R^n$ as column vectors.

\subsection*{Acknowledgments.} 
The first author thanks Indiana University for its hospitality and was supported by the NCCR SwissMap and a Doc.Mobility fellowship, both funded by the Swiss FNS. The second author was supported by the SFB 1085 ``Higher invariants'', funded by the Deutsche Forschungsgemeinschaft (DFG). The third author was supported by the GK ``Curvature, Cycles and Cohomology'', also funded by the Deutsche Forschungsgemeinschaft (DFG).
The authors also with to the thank the referee for carefully reading the first version of this paper.

\section{Topological preliminaries}
\label{sec:preliminaries}

In this section, we review the necessary preliminaries needed for the proof of Theorem \ref{thm:main}. In Subsection \ref{sub:Seifert}, following closely \cite{CF}, we recall the definition of C-complexes and generalized Seifert matrices. In Subsections \ref{sub:twisted} and \ref{sub:intersection forms}, we recall the definition of twisted homology modules and intersections forms. Finally, Subsection \ref{sub:Blanchfield} deals with the Blanchfield pairing.

\subsection{C-complexes and generalized Seifert matrices}
\label{sub:Seifert}

For $\mu \in \mathbb{N}$, a $\mu$-colored link is an oriented link $L$ in $S^3$ whose components are partitioned into $\mu$ sublinks $L_1\cup \dots \cup L_\mu$.  A {\em~$C$-complex\/}~\citep{Co} for a~$\mu$-colored link~$L$ is a union~$S=S_1\cup\dots\cup S_\mu$ of surfaces in~$S^3$ such that:
\begin{romanlist}
\item{for all~$i$,~$S_i$ is a Seifert surface for the sublink of~$L$ of color~$i$;}
\item{for all~$i\neq j$,~$S_i\cap S_j$ is either empty or a union of clasps (see Figure \ref{fig:clasp});}
\item{for all~$i,j,k$ pairwise distinct,~$S_i\cap S_j\cap S_k$ is empty.}
\end{romanlist}

In the case~$\mu=1$, a~$C$-complex for~$L$ is nothing but a Seifert surface for the link~$L$, while the existence of a~$C$-complex for an arbitrary colored link was established in~\cite[Lemma 1]{Cim}. 
\begin{figure}[h]
\labellist\small\hair 2.5pt
\pinlabel {$\alpha$} at 18 58
\pinlabel {$S_i$} at 94 110
\pinlabel {$S_j$} at 331 118
\endlabellist
\centering
\includegraphics[width=0.4\textwidth]{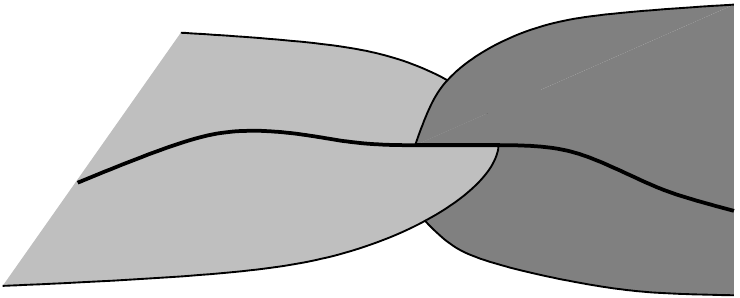}
\caption{A clasp intersection crossed by a~$1$-cycle~$x$.}
\label{fig:clasp}
\end{figure}
Given a sequence~$\eps=(\eps_1,\dots,\eps_\mu)$ of~$\pm 1$'s, let~$i^\eps\colon H_1(S)\to H_1(S^3\setminus S)$ be defined as follows. Any homology class $x\in H_1(S)$ can be represented by an oriented cycle~$\alpha$ which behaves as
illustrated in Figure~\ref{fig:clasp} whenever crossing a clasp. Then, define~$i^\eps(x)$ as
the class of a~$1$-cycle obtained by pushing~$\alpha$ in the~$\eps_i$-normal direction off~$S_i$ for~$i=1,\dots,\mu$. 
Finally, consider the bilinear form
\[
H_1(S)\times H_1(S)\to\Z,\quad(x,y)\mapsto \op{lk} (i^\eps(x),y)\,,
\]
where~$\op{lk}$ denotes the linking number.
Fixing a basis of~$H_1(S)$, the resulting matrices~$A^\eps$ are called \textit{generalized Seifert matrices} of $L$. As in the introduction, we set
$$H(t)\,\,=\,\,\sum_\varepsilon \, \prod_{i=1}^\mu (1-t_i^{\varepsilon_i})\, A^\varepsilon,$$ 
where the sum is on all sequences~$\varepsilon=(\varepsilon_1,\varepsilon_2,\dots, \varepsilon_\mu)$ of~$\pm 1$'s. Since $A^{-\eps}=(A^\eps)^T$, $H(t)$ is a hermitian matrix with respect to the involution $t_i\mapsto t_i^{-1}$. Evaluating $H(t)$ at some $\omega\in (S^1)^\mu$ yields a hermitian matrix $H(\omega)$ with complex coefficients, but note that this is the \textit{complex conjugate} of the matrix with the same name in \cite[Section 2.2]{CF}.

We conclude this subsection by introducing some terminology.

\begin{definition}
\label{def:nice}
A curve on a C-complex $S$ is called \emph{nice} if the following conditions are satisfied:
\begin{enumerate}
\item it has no self-intersections, 
\item the restriction to each component $S_i$ is an embedding,
\item it intersects each clasp at most once,
\item when it intersects a clasp, then it looks locally like in Figure~\ref{fig:clasp}.
\end{enumerate}
\end{definition}

\begin{lemma}\label{lem:nice-basis}
There exists a basis of $H_1(S)$ for which each element is represented by a nice curve.
\end{lemma}

\begin{proof}
Up to homotopy equivalence, $S$ can be constructed by taking the disjoint union of the surfaces $S_i$ and adding an arc connecting $S_i$ with $S_j$ for each clasp. Contracting every surface to a point produces a graph $\Gamma$ with $\mu$ vertices $V_1,\dots, V_\mu$ and one edge for each clasp. This construction yields the short exact sequence
\[ 0\,\to \,\smoplus{i=1}{\mu} H_1(S_i)\,\to\, H_1(S)\,\xrightarrow{\pi} \,H_1(\Gamma)\,\to\,0, \]
where the non-trivial maps are respectively induced by the inclusions of the disjoint $S_i$'s into $S$, and the projection to the quotient. The surjectivity of $\pi$ is immediate since any embedding of $\Gamma$ into $S$ produces a right inverse. 

It is clear that each $H_1(S_i)$ admits a basis given by embedded curves that do not intersect any of the clasps. Thus it remains to find nice curves on $S$ whose images under $\pi$ form a basis for $H_1(\Gamma)$. Next, we say that a path on $\Gamma$ is \textit{simple} if it  intersects each vertex and each edge at most once. Consequently, the lemma will follow from the following two assertions:
\begin{enumerate}
\item $H_1(\Gamma)$ admits a basis consisting of simple closed curves,
\item given any simple closed curve $\gamma$ on $\Gamma$, there exists a nice curve $s$ on $S$ with $\pi(s)=\gamma$.
\end{enumerate}
The first statement is of course well-known. For the reader's convenience we sketch the argument. Let $T$ be a maximal tree in $\Gamma$ and let $e_1,\dots,e_k$ be the edges in $\Gamma\sm T$. We can connect the end points of each $e_i$ by a simple path  $p_i$ in the tree $T$. Now, for $i=1,\dots,k$, the curves $\gamma_i=e_i\cup p_i$ are simple and represent a basis for $H_1(\Gamma)$.

In order to prove the second assertion, observe that each vertex $V_i$ crossed by a simple closed curve $\gamma$ is both the initial point of a unique edge crossed by $\gamma$ and the terminal point of a unique edge crossed by $\gamma$. Next, we pick an embedded curve $\gamma_i$ on the corresponding surface $S_i$ connecting the end points of the two clasps. Finally, we define $s$ as the curve on $S$ which is given by the union of the following paths:
\begin{enumerate}
\item for each edge crossed by $\gamma$, we take the corresponding clasp,
\item for each vertex $V_i$ crossed by $\gamma$, we take the simple closed curve $\gamma_i$ on $S_i$.
\end{enumerate}
Since $s$ is clearly nice and satisfies $\pi(s)=\gamma$, this concludes the proof.
\end{proof}

\subsection{Twisted homology and cohomology groups}
\label{sub:twisted}

Let~$X$ be a CW complex, let~$\varphi\colon \pi_1(X) \rightarrow \mathbb{Z}^\mu$ be an epimorphism, and denote by~$p\colon \widetilde{X} \rightarrow X$ the regular cover of~$X$ corresponding to the kernel of~$\varphi$. Given a subspace~$ Y \subset X$, we shall write~$\widetilde{Y}=p^{-1}(Y)$, and view~$C_*(\widetilde{X},\widetilde{Y})$ as a chain-complex of free left modules over the ring $\Lambda=\mathbb{Z}[t_1^{\pm 1},\dots,t_\mu^{\pm 1}]$. (The ring $\Lambda$ is of course commutative, so any left-module is also a right-module, nonetheless it is quite helpful to keep in mind the ``natural'' module structures which would also work over non-commutative rings.) Given a commutative ring $R$ and a $(R,\Lambda)$-bimodule~$M$, consider the chain complexes
\begin{align*}
&C_*(X,Y;M)=M \otimes_{\Lambda}C_*(\widetilde{X},\widetilde{Y}) \\
&C^*(X,Y;M)=\Hom_{\Lambda}\Big(\overline{C_*\big(\widetilde{X},\widetilde{Y}\big)},M\Big)
\end{align*}
of~left $R$-modules and denote the corresponding homology $R$-modules by~$H_*(X,Y;M)$ and~$H^*(X,Y;M)$.  In particular, one can use the canonical isomorphism of~$\Lambda \otimes_{\Lambda}C_*(\widetilde{X},\widetilde{Y})$ with~$C_*(\widetilde{X},\widetilde{Y})$ to recover the homology of the covering space:
$$ H_*(X,Y;\Lambda) \cong H_*(\widetilde{X},\widetilde{Y};\mathbb{Z}).$$
Since localizations are flat, we obtain that $H_*(X,Y;\Lambda_S)=H_*(X,Y;\Lambda)\otimes_{\Lambda}\Lambda_S$. 
For later use, let us fix some additional notation. Sending a
cocycle~$f \in  \Hom_{\Lambda} \left(  \overline{C_*(\widetilde{X},\widetilde{Y})},Q/\Lambda_S \right)$ to the~$\Lambda_S$-linear map defined by~$\sigma\otimes p \mapsto \overline{f(\sigma)}\cdot p$  yields a well-defined isomorphism of~left $\Lambda_S$-modules  
$$\kappa\colon H^i(X,Y;Q/\Lambda_S) \,\,\rightarrow\,\, H_i\Big(\overline{\Hom_{\Lambda_S}\big(C_*(X,Y;\Lambda_S),Q/\Lambda_S\big)}\Big).$$
We also consider the evaluation map
$$ \text{ev}\colon  H_i\Big(\overline{\Hom_{\Lambda_S}(C_*(X,Y;\Lambda_S),Q/\Lambda_S)}\Big) \,\,\rightarrow\,\, \overline{\Hom_{\Lambda_S}(H_i(C_*(X,Y;\Lambda_S)),Q/\Lambda_S)}.$$
The composition $\text{ev} \circ \kappa$ will allow us to pass from cohomology to the dual of a homology module.  Finally, the following well-known lemma motivates our use of $\Lambda_S$-coefficients:

\begin{lemma}
\label{lem:lambdaS}
Let~$X$ be a connected CW-complex, and let~$\varphi \colon \pi_1(X \times S^1) \rightarrow \mathbb{Z}^\mu$ be a homomorphism.
\begin{enumerate}[font=\normalfont]
\item If the composition~$\pi_1(S^1) \rightarrow \pi_1(X \times S^1) \xrightarrow{\varphi}  \Z^\mu$ sends a generator of~$\pi_1(S^1)$ to a non-trivial element~$z$ of~$\Z^\mu$, then the chain complex~$C_*(X \times S^1;\Z[\Z^\mu][(z-1)^{-1}])$ is acyclic.
\item  If $z$ is a non-trivial element in the image of $\varphi$, then 
$H_0(X;\Z[\Z^\mu][(z-1)^{-1}])=0$.
\end{enumerate}
\end{lemma}

\begin{proof} A proof of the first statement can be found in \cite[Example 2.7]{Ni}), while a proof of the second statement follows from the standard calculation of 0-th twisted homology group, see e.g.\ \cite[Chapter~VI.3]{HS}.
\end{proof}

\subsection{Intersection forms}
\label{sub:intersection forms}

Given a compact connected oriented $n$-manifold $X$, a homomorphism $\varphi \colon \pi_1(X) \rightarrow \mathbb{Z}^\mu$ and a $(\Lambda_S,\Lambda)$-bimodule $M$, Poincar\'e duality defines isomorphisms of $\Lambda_S$-modules 
\begin{align*}
H_i(X,\partial X;M) &\cong H^{n-i}(X;M), \\
H_i(X;M) & \cong H^{n-i}(X,\partial X;M).
\end{align*}
We now consider the following sequence of $\Lambda_S$-homomorphisms
$$  \Theta \colon H_i(X;M) {\rightarrow} H_i(X,\partial W;M) \xrightarrow{\operatorname{PD}} H^{n-i}(X;M) \xrightarrow{\ev \circ \kappa} \overline{\Hom_{\Lambda_S}(H_{n-i}(X;M),\Lambda_S)}. $$
Here the first map is induced by the inclusion, the second is Poincar\'e duality, and the last was described in Subsection \ref{sub:twisted}. Assuming $X$ is $4$-dimensional, the \textit{intersection form} on $H_2(X;M)$ is defined as $\lambda(x,y)=\Theta(y)(x)$. More explicitly, if $M=\Lambda_S$ and if $a, b$ are elements of $H_2(X;M)$, then the twisted intersection form is given by $$\lambda(a,b)= \sum_{g\in \mathbb{Z}^{\mu}} s(a, t^g \, b) \, t^g \in \Lambda_S,$$
where $s$ denotes the ordinary intersection number on $\widetilde{X}$. Notice that $\lambda$ is hermitian and sesquilinear over~$\Lambda_S$, in the sense that~$\lambda(a,b)=\overline {\lambda(b,a)}$ and $\lambda(pa,qb)=p\lambda(a,b)\overline{q}$ for any~$a,b \in H_2(X;M)$ and~$p,q \in \Lambda_S.$

\subsection{The Blanchfield pairing of a colored link}
\label{sub:Blanchfield}

Let $L=L_1 \cup \dots \cup L_\mu$ be a colored link and let $X_L$ denote its exterior. Identifying $\mathbb{Z}^\mu$ with the free abelian group on $t_1,\dots,t_\mu$ and precomposing the epimorphism~$ H_1(X_L) \rightarrow \mathbb{Z}^\mu, \gamma \mapsto t_1^{\text{lk}(\gamma,L_1)}\cdots t_\mu^{\text{lk}(\gamma,L_\mu)}$ with abelianization gives rise to the \textit{Alexander module} $H_1(X_L;\Lambda_S)$ of $L$. 
Assume that~$H_1(X_L;\Lambda_S)$ is torsion over~$\Lambda_S$,  and denote by~$\Omega$ the following composition of $\Lambda_S$-homomorphisms:
\[ 
\begin{array}{rcl}
 H_1(X_L;\Lambda_S) \,\xrightarrow{\cong} \,H_1(X_L,\partial X_L;\Lambda_S) 
 \, \xrightarrow{\text{PD}}\, H^2(X_L;\Lambda_S) \,
 & \xrightarrow{\text{BS}^{-1}}&  H^1(X_L;Q/\Lambda_S) \\
 & \xrightarrow{\text{ev} \circ \kappa}& \overline{\Hom_{\Lambda_S}(H_1(X_L;\Lambda_S),Q/\Lambda_S)},
 \end{array}\]
where~$PD$ is Poincar\'e  duality and~$BS^{-1}$ is the inverse of the Bockstein homomorphism arising from the short exact sequence
~$$ 0 \rightarrow \Lambda_S \rightarrow Q \rightarrow Q/\Lambda_S \rightarrow 0$$
of coefficients. Note that in this step we implicitly used that $H_1(X_L;\Lambda_S)$ is $\Lambda_S$-torsion: indeed this assumption implies by Poincar\'e duality and an Euler characteristic argument that $H^1(X_L;Q)=H^2(X_L;Q)=0$, which in turn shows that $BS$ is invertible. It is a consequence of Lemma~\ref{lem:lambdaS} and the long exact sequence of the homology modules of the pair $(X_L,\partial X_L)$ with $\Lambda_S$-coefficients that the first homomorphism is an isomorphism. The second map is evidently an isomorphism and we already explained why the Bockstein map is invertible. Finally, the last map is also an isomorphism. This follows from the Universal Coefficient Spectral
 Sequence~\cite[Theorem~2.3]{Lev} and the fact, see Lemma~\ref{lem:lambdaS}~(2), that $H_0(X_L;\Lambda_S)=0$.

\begin{definition}
The \textit{Blanchfield pairing} of a colored link~$L$ with torsion Alexander module is the pairing
$$ \operatorname{Bl}(L) \colon  H_1(X_L;\Lambda_S) \times  H_1(X_L;\Lambda_S) \rightarrow Q/\Lambda_S ~$$
defined by~$\operatorname{Bl}(L)(a,b)=\Omega(b)(a).$
\end{definition}

It follows from the definitions that the Blanchfield pairing is sesquilinear over~$\Lambda_S$, in the sense that~$\operatorname{Bl}(L)(pa,qb)=p\operatorname{Bl}(L)(a,b)\overline{q}$ for any~$a,b \in H_1(X_L;\Lambda_S)$ and~$p,q \in \Lambda_S.$
Furthermore, the above discussion shows that this pairing is non-singular. 

\section{Pushed-in C-complexes}
\label{sec: Pushin}

In this section, we define the notion of a ``pushed-in C-complex" in the 4-ball $D^4$, study its exterior (Subsection \ref{sub:complement}) and compute its fundamental group (Subsection \ref{sub:fund}). Note that our approach differs slightly from the existing literature \cite{CF,CooperThesis}: instead of only pushing in the interiors of the Seifert surfaces, we also push in radially the corresponding sublinks. Moreover the different Seifert surfaces end up at different depths of the $4$-ball. 

\subsection{The complement of a pushed-in C-complex}
\label{sub:complement}
Let $S=S_1\cup \dots\cup S_\mu$ be a C-complex for a $\mu$-colored link $L$ and view $S^3$ as the boundary of $D^4$. For $i=1,\dots,\mu,$ we pick a tubular neighborhood $S_i\times [-2,2]$ of $S_i$ in $S^3$. Furthermore for each $i$, we fix two surfaces with boundary $S_i'$, $S_i''$ such that $S_i'\subset S_i\subset S_i''$, the complement $S_i''\sm \interior(S_i')$ is a union of small annuli around $L_i=\partial S_i$, and the respective unions $S'$ and $S''$ form C-complexes for links isotopic to $L$. Let us fix once and for all an embedding of $S^3\times [0,\mu]$ in $D^4$ such that $S^3\times \{0\}$ agrees with $S^3=\partial D^4$. In order to prevent the different tubular neighborhoods from getting mixed up, we denote the image of $(p,t)$ under this map by $p\star t$. For $i=1,\dots,\mu,$ we write 
\[ F_i\,:=\,L_i\star\big[0,i-\tmfrac{1}{4}\big]\cup_{L_i \star \lbrace i-\frac{1}{4} \rbrace} S_i\star \big\{i-\tmfrac{1}{4}\big\}\]
 and refer to 
$F:=F_1\cup \dots \cup F_\mu$
 as the \textit{push-in} of $S$. In other words, $F$ is obtained by pushing each sublink $L_i$ radially (at a different depth) into the 4-ball and then capping it off with $S_i$. Observe that the $F_i$ intersect pairwise in double points and consequently $F$ has boundary $L$ (in the sense of \cite[Section 6]{CF}). Since our goal is to study the exterior of $F$ in $D^4$, we define
$$ \nu F\,:=\, \lmcup{i=1}{\mu} \Big((\interior(S_i'')\sm S_i')\times(-1,1)\star [0,i) \,\,\cup \,\,S_i\times(-1,1)\star (i-\tmfrac{1}{2},i)\Big),$$
 and
$W_F:= D^4\sm \nu F.$
In order to compute the homology of $W_F$, we shall now decompose the latter space into more manageable pieces. First of all, denote by 
	\[ B\,\,:=\,\,D^4\,\sm\, \left(\smcup{i=1}{\mu}\,\interior(S_i'')\times (-1,1)\star[0,i) \right )\] 
 the complement of the whole trace of the push-in and observe that $B$ is homeomorphic to a 4-ball. In order to recover $W_F$ from $B$, we first set
	\[ \hspace{1cm} Y_i\,\,:=\,\,S_i'\sm \Big(\ttmcup{j\ne i}{} \interior(S_j'')\times (-2,2)\Big)\star[0,i-1]
	\,\cup\,
	S_i'\sm \Big(\ttmcup{j=i+1}{\mu}  \interior(S_j'')\times (-2,2)\Big)\star [i-1,i-\tmfrac{1}{2}]\]
for $i=1,\dots,\mu$. Observe that $Y_i$ is a closed subset of $S_i'\star [0, i-\tmfrac{1}{2}]$ which is homotopy equivalent to $S_i$. Moreover, making use of the neighborhoods of the $S_i$ in $S^3$, it makes sense to consider the closed sets $Y_i\times [-1,1] \subseteq D^4$. In the definition of the $Y_i$'s, we removed large enough neighborhoods of the clasps in order to make these sets disjoint. It remains to add the clasp parts. For $i<j$, we define the space
	\[ X_{ij}\,\,:=\,\,(S_i'\cap 	S_j')\, \star\, [0, i- \tmfrac{1}{2}]\]
which consists in a disjoint union $\bigsqcup X_{ij}^k$ of 2-disks, one for each clasp between $S_i$ and $S_j$. Using the slightly larger neighborhoods of $S_i$ and $S_j$ in $S^3$, we consider the cross-shaped subset of $[-2,2]\times [-2,2]$ given by
 $$K:=[-1,1]\times [-2,2] \cup [-2,2]\times [-1,1].$$
This way, the space $W_F$  decomposes as $B \cup Z$, where 
$$ Z\,:=\,\smsqcup{i=1}{\mu} Y_i\times[-1,1]\,\,\cup\,\, \smsqcup{i<j}{}X_{ij}\times K.$$
Observe that for each nice curve $\alpha$ in $S$ (recall Definition \ref{def:nice}), the push-offs $i^\varepsilon([\alpha])$ described in Section \ref{sub:Seifert} can be represented by curves $\alpha^\varepsilon$ which are embedded in the intersection of $S^3$ with $\bigcup_{i=1}^\mu Y_i \times \{\pm 1\} \subseteq \partial B$.

\subsection{The fundamental group of $W_F$}
\label{sub:fund}
Given a C-complex $S$, denote by $J$ the subset of $\{1,\dotsc , \mu\}^2$ consisting of pairs $(i,j)$ for which $i<j$ and there exists at least one clasp between the surfaces $S_i$ and $S_j$.

\begin{proposition}
\label{prop:pi1}
The fundamental group of $W_F$ admits the presentation
	\begin{equation*}
\langle a_1,\dotsc , a_{\mu} \, \vert \, [a_i,a_j]=e \:\text{ for all } (i,j) \in J \rangle, 
	\end{equation*}
	where the generators $a_1,\dotsc , a_{\mu}$ correspond to meridians for the surfaces $F_1,\dotsc , F_{\mu}$.
\end{proposition}

\begin{proof} 
Recall from Subsection \ref{sub:complement} that $W_F=B \cup Z$, where $B$ is contractible and $Z=\bigsqcup_{i=1}^\mu Y_i\times[-1,1]\,\,\cup\,\, \bigsqcup_{i<j}X_{ij}\times K$. Observe that gluing $Y_1\times [-1,1]$ to $B$ is homotopically the same as identifying $Y_1\times \{-1\}$ with $Y_1\times \{1\}$ so that
	\begin{equation*}
	\pi_1(B\cup (Y_1\times [-1,1]))\cong \langle a_1 \, \vert \, a_1 \cdot 1 \cdot a_1^{-1}=1\rangle = \langle a_1 \rangle.
	\end{equation*}
	The generator $a_1$ is a meridian for the surface $F_1$.
Gluing successively $Y_2\times [-1,1],\dotsc , Y_{\mu}\times [-1,1]$ and repeating the argument adds one new generator $a_i$ for each $i$, namely the meridian of the surface $F_i$. As each inclusion induced map $\pi_1(Y_i\times \{\pm 1\}) \rightarrow \pi_1(Y_i \times [-1,1])$ factors through the trivial group $\pi_1(B)$, no relations are added and thus
$$ 	\pi_1\bigg(B\cup \smsqcup{i=1}{\mu} Y_i\times[-1,1]\bigg)\,\,\cong \,\, \langle a_1,\dotsc , a_{\mu}  \rangle. $$
In order to recover $W_F$, it remains to glue back in  the  contractible ``clasp parts'' $X_{ij}^k\times K$ (recall that $X_{ij}=\bigsqcup_k X_{ij}^k$), which are only non-empty when $(i,j)\in J$.  Note that $X_{ij}^k\times K$ and $B\cup \bigsqcup_{i=1}^{\mu} Y_i\times[-1,1]$  intersect in  $P:= X_{ij}^k \times \partial K$ which is homotopy equivalent to a circle. Moreover, under the inclusion map of $P$ into $B\cup \bigsqcup_{i=1}^{\mu} Y_i\times[-1,1]$, a generator of $\pi_1(P)$ is sent (up to inversion) to a commutator of the form $[a_i,a_j]$. Hence, by Van Kampen's theorem, one gets
	\begin{equation*} 
	\pi_1 \bigg( B \cup \smsqcup{i=1}{\mu} Y_i\times[-1,1] \cup (X_{ij}^k \times K)\bigg) \,\,\cong \,\,\langle a_1,\dotsc , a_{\mu}  \, \vert \, [a_i,a_j] \rangle .
	\end{equation*}	
Repeating the process for each $X_{ij}^k$ immediately yields the proposition.
\end{proof}

Let $p \colon \overline{W}_F \to W_F$ be the cover of $W_F$ corresponding to the kernel of the abelianization map $\pi_1(W_F) \rightarrow H_1(W_F)$ (see Figure \ref{fig:cover} for a schematic picture). Recall from the introduction that a C-complex $S$ is \textit{totally connected} if each $S_i$ is connected and $S_i \cap S_j  \neq \emptyset $ for all $i\neq j$. Proposition \ref{prop:pi1} implies both that the deck transformation group of $\overline{W}_F$ is free abelian of rank $\mu$ and the following result.

\begin{corollary}
\label{cor:H10}
If the C-complex $S$ is totally connected, then $H_1(W_F;\Lambda_S)$ vanishes.
\end{corollary}

\begin{proof}
If each pair of surfaces in $S$ is joined by at least one clasp, then it follows from Proposition~\ref{prop:pi1} that $\pi_1(W_F)$ is the free-abelian group on $t_1,\dots,t_\mu$. This implies that $H_1(W_F;\Lambda)=0$, which by flatness of $\Lambda_S$ also implies that $H_1(W_F;\Lambda_S)=0$.
\end{proof}

\section{The twisted intersection forms of C-complex exteriors in the 4-ball.}
\label{sec:intform}

In this section, we describe an explicit isomorphism $H_1(S) \otimes \Lambda_S\to H_2(W_F;\Lambda_S)$  (Subsection \ref{sub:H2}) and compute the corresponding intersection form (Subsection \ref{sub:intform}).

\subsection{A geometric basis for $H_2(W_F;\Lambda_S)$}
\label{sub:H2}

The goal of this subsection is to provide a convenient basis for the $\Lambda_S$-module $H_2(W_F;\L_S)$. In order to state the main result, we briefly introduce some notation. Given an arbitrary lift $\overline{B}$ of $B$ to $\overline{W}_F$ and a nice curve $\alpha$ on $S$, lift each push-off $\alpha^\varepsilon$ to $\overline{B}$, and call it $\overline{\alpha}^\varepsilon \subset \overline{B}$. Also, set
\begin{equation*}
\sgn{\varepsilon}:= \prod_{i=1}^\mu \varepsilon_i\in \{\pm 1\}, \:\:\:\:\:\:\:\:\:\:\: 
\pmb{1}:=(1,\dotsc ,1)\in \mathbb{Z}^\mu,
\end{equation*}
and notice that the following equality holds:
\begin{equation}\label{eq:epsilonformula}
\sum_{\varepsilon}\sgn \varepsilon \: t^{\frac{\pmb{1}+\varepsilon}{2}}= \prod_{i=1}^\mu (t_i-1).
\end{equation}
Next, using Lemma \ref{lem:nice-basis}, we fix once and for all a basis $\mathcal{B}$ for $H_1(S)$ so that each element of $\mathcal{B}$ is represented by a nice curve, resulting in a set $\textbf{B}$ of representatives. The remainder of this subsection will be devoted to the proof of the following result. 
\begin{proposition}
\label{prop:phi}
For each nice curve $\alpha$ on $S$, there is a closed surface $\Phi_\alpha$ embedded in $\overline{W}_F$, which intersects $p^{-1}(B\cap Z)$ in the curve $ \sum_\varepsilon \sgn(\varepsilon)t^{\frac{\pmb{1}+\varepsilon}{2}}\overline{\alpha}^\varepsilon$. The map
\[
\Phi\colon H_1(S)\otimes \Lambda_S \to H_2(W_F; \Lambda_S) 
\]
defined on the elements of $\mathbf{B}$ as $[\alpha]\otimes 1 \mapsto [\Phi_\alpha]$ is an isomorphism
of $\Lambda_S$-modules.
\end{proposition}

\subsubsection{The construction of the surfaces $\Phi_{\alpha}$}
\label{subsub:phi}
Given a nice cycle $\alpha \subset S$, there is a Seifert surface for each $\alpha^{\varepsilon}$ (viewed as an oriented knot) in $\partial B\cong S^3$. Pushing the interior of these surfaces inside $B$ provides properly embedded surfaces $S_\alpha^\varepsilon$ in $B$ whose boundary is $\alpha^{\varepsilon}$. Lifting these surfaces to $\overline{W}_F$ as subsets $\overline{S}_\alpha^\varepsilon$ of the fixed lift $\overline{B}$ of $B$, one has $$\partial (t^{\frac{\pmb{1}+\varepsilon}{2}}\overline{S}_{\alpha}^{\varepsilon})= t^{\frac{\pmb{1}+\varepsilon}{2}}\overline{\alpha}^{\varepsilon}.$$
In order to build a closed surface from all of these disjoint $t^{\frac{\pmb{1}+\varepsilon}{2}} \overline{S}_\alpha^\varepsilon$, decompose $\alpha$ as
\begin{equation} 
\label{eq:curvedec}
	\alpha\,\,=\,\, \bigg(\smcup{i=1}{\mu} \alpha_i \bigg) \,\,\cup \,\, \bigg( \smcup{i<j}{}\, \smcup{k=1}{c(i,j)}\, \alpha_{ij}^k \bigg),
\end{equation}
where $\alpha_{i}$ is the (possibly empty) subset of $\alpha$ which lies in $S_i \setminus \bigcup_{j \neq i}{S_j}$ and $\alpha_{ij}^k$ is the (possibly empty) subset of $\alpha$ corresponding to the clasp indexed by the triple $(i,j,k)$. One can perform analogous decompositions for each push-off $\alpha^\varepsilon$ yielding segments $\alpha_i^\varepsilon$ and $\alpha_{ij}^{k,\varepsilon}$. Given two sequences $\varepsilon, \varepsilon'$ which differ only at the index $j$, say with $\varepsilon_j=-1$ and $\varepsilon'_j=+1$, connect the two surfaces $t^{\frac{\pmb{1}+\varepsilon}{2}}\overline{S}_{\alpha}^{\varepsilon}$ and  $-t^{\frac{\pmb{1}+\varepsilon'}{2}}\overline{S}_{\alpha}^{\varepsilon'}$ by adding a cylinder $t^{\frac{\pmb{1}+\varepsilon}{2}}\overline{\alpha}_j\times [-1,1]\subseteq p^{-1}(Y_j\times[-1,1])$. Repeating this process for all $\varepsilon$, $\varepsilon'$ as above, we can now set
$$ S_\alpha\,\,:=\,\,\lmcup{\varepsilon}{} \bigg( \sgn(\varepsilon) t^{\frac{\pmb{1}+\varepsilon}{2}}\overline{S}_\alpha^\varepsilon \,\,\cup\,\, \smcup{ \lbrace j\, \vert \, \varepsilon_j=-1 \rbrace}{}\sgn(\varepsilon) t^{\frac{\pmb{1}+\varepsilon}{2}} \overline{\alpha}_j \times [-1,1] \bigg), $$
where the sign $\sgn(\varepsilon)=\varepsilon_1\cdots\varepsilon_{\mu}$ is added for the orientations to be consistent. We see from its construction that $S_\alpha$ is a surface whose boundary lies in the boundary of the union of (lifts of) the topological 4-balls $X_{ij}^k\times K$. Consequently, one can find a surface $S^{\clasp}$ whose components lie in those 4-balls and such that $\partial S_{\alpha} ^{\clasp} = - \partial S_{\alpha}$.
We can hence define a closed surface
		$$ \Phi_\alpha :=S_ \alpha \cup_\partial S_\alpha^{\clasp}.$$
Using our set $\mathbf{B}$ of nice representatives 
for the basis $\mathcal{B}$ of $H_1(S)$, we can now define a $\Lambda_S$-linear map $$ \Phi\colon H_1(S)\otimes \Lambda_S \to H_2(W_F; \Lambda_S)$$ by $[\alpha]\otimes 1 \mapsto [\Phi_{\alpha}]$ for $\alpha\in\mathbf{B}$. From the construction of $\Phi$, we get the following formula.

	\begin{proposition}
	\label{prop:intersects}
Let $\partial \colon H_2(W_F;\Lambda_S) \rightarrow H_1(B \cap Z;\Lambda_S)$ be the boundary map in the Mayer-Vietoris sequence of $W_F=B \cup Z$. Then
\begin{equation}
\label{eq:intersects}
\partial \Phi ([\alpha]\otimes 1)= \sum_\varepsilon \sgn(\varepsilon)t^{\frac{\pmb{1}+\varepsilon}{2}}[{\overline\alpha}^\varepsilon]
\end{equation} 
for each nice cycle $\alpha$.
\end{proposition}

Our goal is now to prove that $\Phi$ is an isomorphism.

\subsubsection{Reducing the problem to a commutativity statement}
\label{subsub:reduc}
Recall from Subsection \ref{sub:complement} that we decomposed $W_F$ into $B \cup Z$, where $Z:=\bigsqcup_{i=1}^\mu Y_i\times[-1,1]\,\,\cup\,\, \bigsqcup_{i<j}X_{ij}\times K$.  We now set 
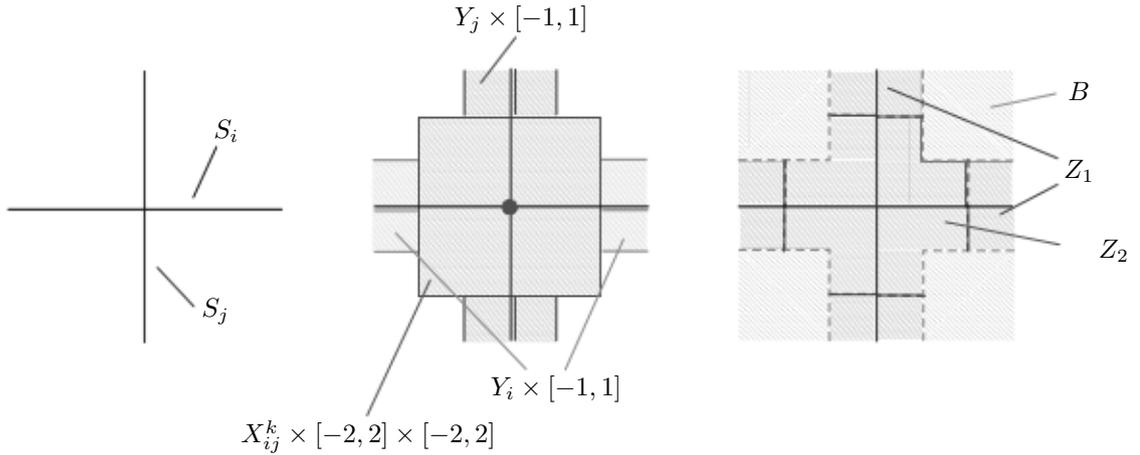
\begin{figure} [b]
	\labellist\small\hair 2.5pt
	\pinlabel {$S_j$} at 78 57
	\pinlabel {$S_i$} at 82 125
	\pinlabel {$X_{ij}^k\times [-2,2]\times [-2,2]$} at  135 10
	\pinlabel {$Y_i\times [-1,1]$} at  205 28
	\pinlabel {$Y_j\times [-1,1]$} at  192 167
	\pinlabel {$Z_2$} at  413 80
	\pinlabel {$Z_1$} at  400 110
	\pinlabel {$B$} at  400 140
	\endlabellist
	\centering
	\includegraphics{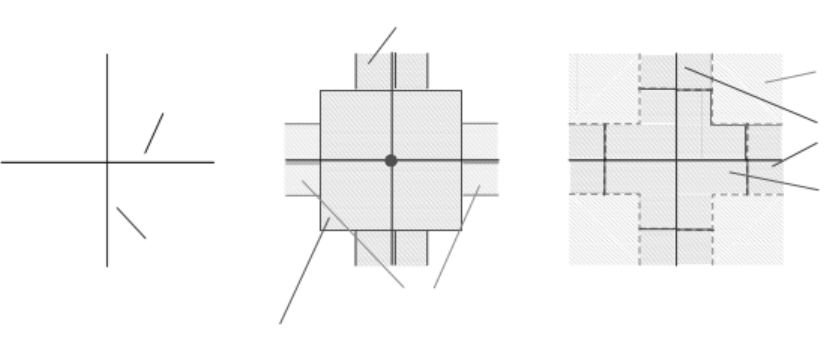}
	\caption{A dimensionally reduced sketch of $Z_1$ and $Z_2$ around a clasp.}
\end{figure}
$$Z_1\,:=\,\smsqcup{i=1}{\mu} Y_i \times [-1,1], \ \ \ \ Z_2\,=\,\smsqcup{i <j }{} X_{ij} \times K$$
so that three applications of the Mayer-Vietoris exact sequence (together with the fact that $Z_2$ is made of contractible components) produce the following commutative diagram:
	\begin{equation}\label{eq:bigdiagram} \xymatrix@C0.4cm@R0.1cm 
	{
		&&
		0\ar[dd]  & 0\ar[dd]&\\
		&&\mbox{\phantom{a}}\\
		&& 
		H_1(B \cap Z_1; \Lambda_S)\ar[dd]^a\ar[r]^\psi   & 
		H_1(Z_1; \Lambda_S) \ar[dd]&\\
		&&\mbox{\phantom{a}}\\
		0\ar[r]&H_2(W_F;\L_S)\ar[r]^\partial&H_{1}(B\cap Z;\L_S)\ar[r]^{\varphi}\ar[dd]^b& H_{1} (Z;\L_S) \ar[dd]
		\\
		&&\mbox{\phantom{a}}\\
		&&
		H_0((B \cap Z_1) \cap (B \cap Z_2); \Lambda_S)\ar[r]^{\ \ \ \ \ \ \ \  L}\ar[dd]^M   & H_0(Z_1 \cap Z_2; \Lambda_S)\ar[dd]^N&\\
		&&\mbox{\phantom{a}}\\
		&& 
		H_0(B \cap Z_1; \Lambda_S) \oplus H_0 (B \cap Z_2; \Lambda_S) \ar[r] & H_0( Z_1; \Lambda_S) \oplus H_0 (Z_2; \Lambda_S).    \\
	}
	\end{equation}

The next lemma provides a first step towards the understanding of $\ker(\varphi)$.

	\begin{lemma}
	\label{lem:snake}
The sequence 
\begin{equation}\label{shortex}
0\to \ker\psi \to \ker \varphi \to \ker L \cap \ker M \to 0
\end{equation}
of $\Lambda_S$-modules is exact.
	\end{lemma}
	\begin{proof}
The previous commutative diagram restricts to
	\[ \xymatrix@C1.4cm@R0.05cm
		{	
			0\ar[dd]  & 0\ar[dd]\\ 
					&&\mbox{\phantom{a}}\\
			H_1(B\cap Z_1; \Lambda_S)\ar[dd]\ar[r]^-{\psi}   & H_1(Z_1; \L_S) \ar[dd]&\\
							&&\mbox{\phantom{a}}\\
	H_{1}(B\cap Z; \L_S)\ar[r]^-{\varphi}\ar[dd]^b & H_{1} (Z;\L_S) \ar[dd] \\
					&&\mbox{\phantom{a}}\\
			\ker(M)\ar[r]^L\ar[dd]^M   &\ker(N) \ar[dd]^N& \\
							&&\mbox{\phantom{a}}\\
			0 & 0.}
		\]	
Applying the snake lemma produces the long exact sequence 
\begin{equation}\label{snake}
0 \to \ker \psi \to \ker \varphi \to \ker L \cap \ker M \to \coker \psi \to \coker \varphi  \to \ker N / \im L\vert_{\ker M} \to 0.
\end{equation}
Since $Z_1= \bigsqcup_{i=1}^\mu Y_i \times [-1,1]$ and $B \cap Z_1=  \bigsqcup_{i=1}^\mu Y_i \times \lbrace \pm 1 \rbrace$, $\psi$ is clearly surjective and the result follows.
	\end{proof}
	
Let us describe the strategy we shall use in order to show that the map $\Phi \colon H_1(S)\otimes \Lambda_S \to H_2(W_F;\Lambda_S)$ defined in Subsection \ref{subsub:phi} is an isomorphism. Using the short exact sequence
\begin{equation} \label{eq:sequenceS}
0\to \bigoplus _{i=1}^\mu H_1(S_i) \xrightarrow{\iota} H_1(S) \xrightarrow{\pi} H_1(\Gamma)\to 0,
\end{equation}
 used in the proof of Lemma \ref{lem:nice-basis}, we shall define isomorphisms $\sigma$ and $\tau$ that fit into a commutative diagram 
\begin{equation}
\label{eq:Comut}
\xymatrix@R0.5cm@C1cm{ 
	0 \ar[d]& &0 \ar[d]\\
	\bigoplus_{i=1}^\mu H_1(S_i) \otimes \Lambda_S \ar[rr]^-\sigma \ar[d]^\iota & &\ker(\psi) \ar[d]^a \\
	H_1(S) \otimes \Lambda_S \ar[d]^\pi  \ar[r]^\Phi  &H_2(W_F;\Lambda_S) \ar[r]^{\partial,\cong}  &\ker(\varphi) \ar[d]^b \\
	H_1(\Gamma) \otimes \Lambda_S \ar[rr]^\tau \ar[d]& &\ker(L) \cap \ker(M)  \ar[d]\\
	0  & & 0.\\
}
\end{equation}
The 5-lemma will then immediately imply that $\Phi$ is an isomorphism.

\begin{figure}[b]
\labellist\small\hair 2.5pt
\pinlabel {$\overline{B}$} at 190 400
\pinlabel {$t_2\overline{B}$} at 190 620
\pinlabel {$t_1\overline{B}$} at 405 400
\pinlabel {$t_1t_2\overline{B}$} at 405 620
\endlabellist
\centering
\includegraphics[width=0.4\textwidth]{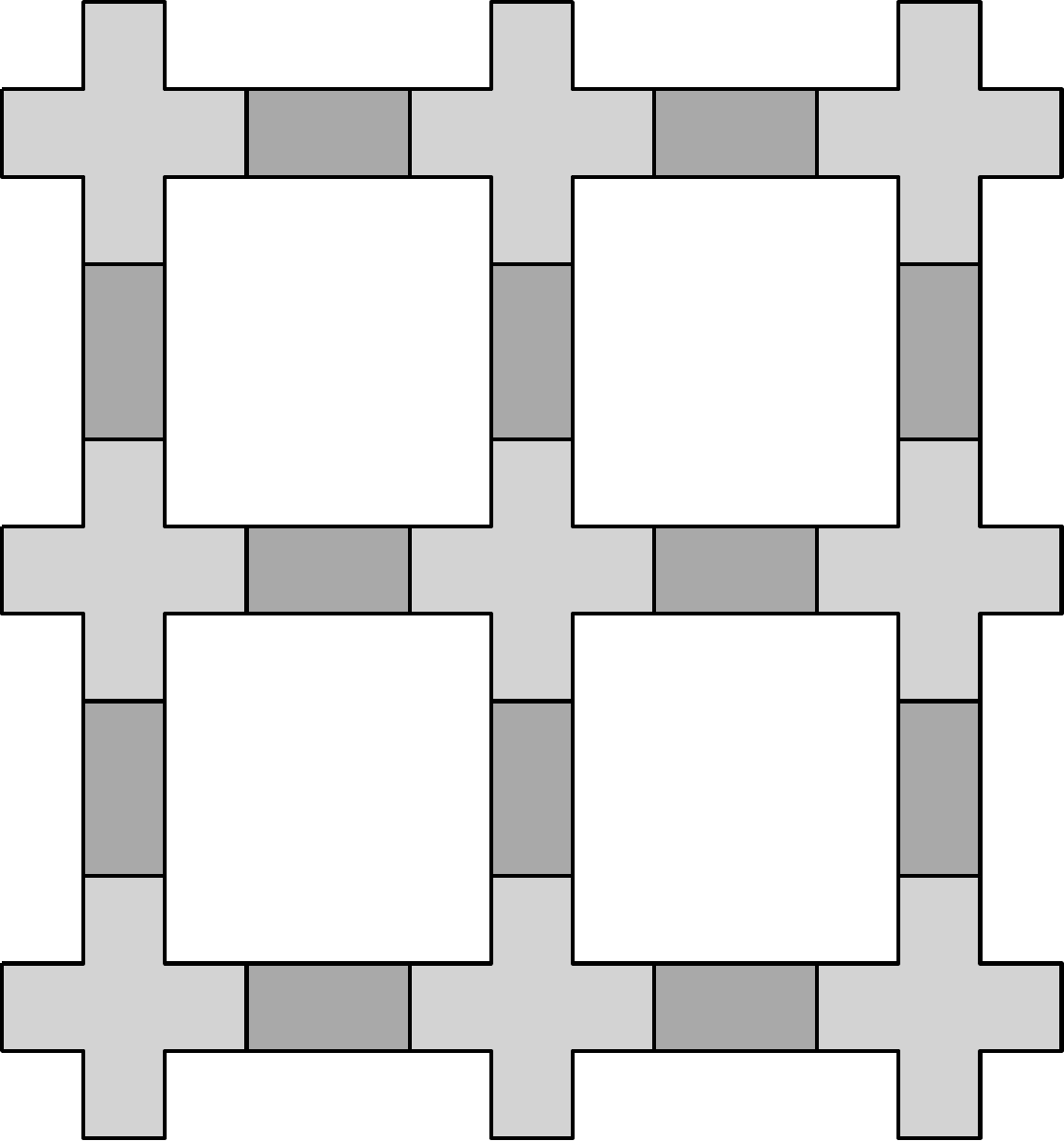}
\caption{A schematic picture of a small portion of the cover $\overline{W}_F$, in the simple case of only two surfaces and one clasp. We represented $p^{-1}(B)$ in white, $p^{-1}(Z_1)$ in dark gray and $p^{-1}(Z_2)$ in light gray.}
\label{fig:cover}
\end{figure}
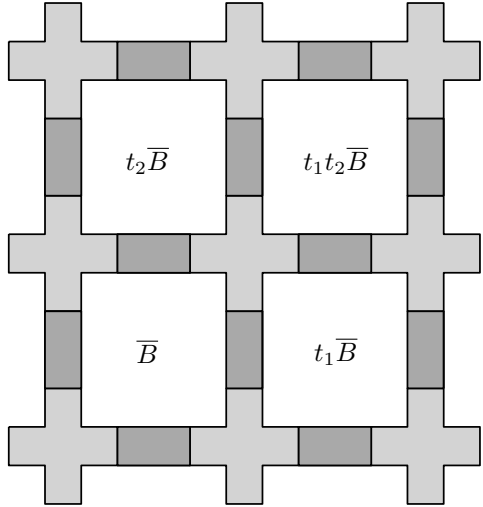

\subsubsection{The short exact sequence of $H_1(S)$}      
\label{subsub:decompoH1(S)}
We fix some notation and recall the short exact sequence which was used in the proof of Lemma \ref{lem:nice-basis}. Let $c(i,j)$ be the number of clasps between surfaces $S_i$ and $S_j$ in the C-complex $S$. The clasps will be denoted by $C_{ij}^k$ for $k=1,\dots,c(i,j)$ resulting in a total number $c$ of clasps. Up to homotopy equivalence, $S$ can be constructed by taking the disjoint union of the surfaces $S_i$ and adding an arc connecting $S_i$ with $S_j$ for each clasp $C_{ij}^k$. Contracting every surface to a point produces a graph $\Gamma$ with $\mu$ vertices $\{V_k\}$ and $c$ edges $\{E_{ij}^k\}$. We shall consider $\Gamma$ as an oriented graph, where the edge $E_{ij}^k$ travels from $V_i$ to $V_j$ if $i<j$. This construction yields the short exact sequence \eqref{eq:sequenceS}, where the non-trivial maps are respectively induced by the inclusions of the disjoint $S_i$'s into $S$, and the projection to the quotient. To get the left vertical sequence of \eqref{eq:Comut}, we just tensor \eqref{eq:sequenceS} with $\Lambda_S$, without changing the name of the maps.

\subsubsection{Constructing the map $\sigma$}	
\label{subsub:sigma}
Recall from Subsection \ref{subsub:reduc} that $W_F=B \cup Z$ was further decomposed by observing that $Z=Z_1 \cup Z_2$ with $Z_1= \bigsqcup_{i=1}^\mu Y_i \times [-1,1]$ and $B \cap Z_1=  \bigsqcup_{i=1}^\mu Y_i \times \lbrace \pm 1 \rbrace$. Consequently, if we lift $B \cap Z$ as a subspace of $B$, then the restriction of $\psi$ to the $k$-th summand of $H_1(B \cap Z_1;\Lambda_S)$ is the map $(x,y) \mapsto x\otimes t_k + y\otimes 1$. Since $S_k$ and $Y_k$ are homotopy equivalent, it follows that
		\begin{equation*}
		\ker \psi = \bigoplus\limits_{i=1}^\mu 
		\Lambda_S\{(-x\otimes 1,x\otimes t_i ) \, \vert \, x\in H_1(S_i)\}.
		\end{equation*}
One can now define the map $\sigma\colon  \bigoplus_{i=1}^\mu H_1(S_i)\otimes \Lambda_S \to \ker \psi$ by 
$$ \sigma(x)\,=\, \bigg(\smprod{j\neq k}{} (t_j-1)\bigg) (-x\otimes 1, x \otimes t_k)$$
for $x \in H_1(S_k)$. Using the description of $\ker(\psi)$, the map $\sigma$ is well defined; it is an isomorphism since the $(t_j-1)$'s are invertible in $\Lambda_S$. We conclude this paragraph by proving the commutativity of the top part of (\ref{eq:sequenceS}). Given a primitive element $x \in H_1(S_k)$, we represent $\iota x\in H_1(S)$ by a nice cycle $\alpha$, which only belongs to $S_k$. Proposition \ref{prop:intersects} hence gives us
$$\partial \Phi \iota x \,\,= \,\, \sum_\varepsilon \sgn(\varepsilon)t^{\frac{\pmb{1}+\varepsilon}{2}}[\overline{\alpha}^\varepsilon].$$
Notice that, since $\alpha$ is contained in $S_k$, the curve $\alpha^\varepsilon$ only depends on the value of $\varepsilon_k$. We thus only have two push-offs for it, which we denote by $\alpha^+$ and $\alpha^-$, as one does in the case when  $S$ consists in a single Seifert surface. We will also denote by  $\hat{\varepsilon}$ the element in $\{\pm 1\}^{\mu-1}$  obtained from $\varepsilon$ by ignoring $\varepsilon_k$. We can now rewrite the above formula as
$$
\partial \Phi \iota x =
\sum_{\varepsilon\in\{\pm 1\}^{\mu}}  \!\!\!\sgn(\hat{\varepsilon}) \prod_{j\ne k}t_j^{\frac{1+\varepsilon_j}{2}} \varepsilon_k t_k^{\frac{1+\varepsilon_k}{2}}[\overline{\alpha}^{\varepsilon_k}] =
	\sum_{\varepsilon'\in\{\pm 1\}^{\mu-1}} \!\!\!\!\! \sgn(\varepsilon')  \prod_{j\ne k}t_j^{\frac{1+\varepsilon_j}{2}} (-[\overline{\alpha}^-]+t_k[\overline{\alpha}^+]).
$$
Applying \eqref{eq:epsilonformula} to the sum over the $\varepsilon'$, we get
$$\partial \Phi \iota x\,\, =\,\, \bigg(\smprod{j\neq k}{} (t_j-1)\bigg)(-[\overline{\alpha}^-]+t_k[\overline{\alpha}^+]).
$$
Commutativity follows, because by definition of $\sigma$ and $a$, we also have
 $$a \sigma(x)\,\,= \,\,\bigg(\smprod{j\neq k}{} (t_j-1)\bigg)a(-x \otimes 1, x \otimes t_k)\,\,=\,\, \bigg(\smprod{j\neq k}{} (t_j-1)\bigg)(-[\overline{\alpha}^-]+t_k[\overline{\alpha}^+]).
$$

\begin{figure}[t]
\labellist\small\hair 2.5pt
\pinlabel {$L$} at 325 759
\pinlabel {$M''$} at 340 550
\pinlabel {$M'$} at 80 550
\endlabellist
\centering
\includegraphics[width=0.7\textwidth,scale=0.6]{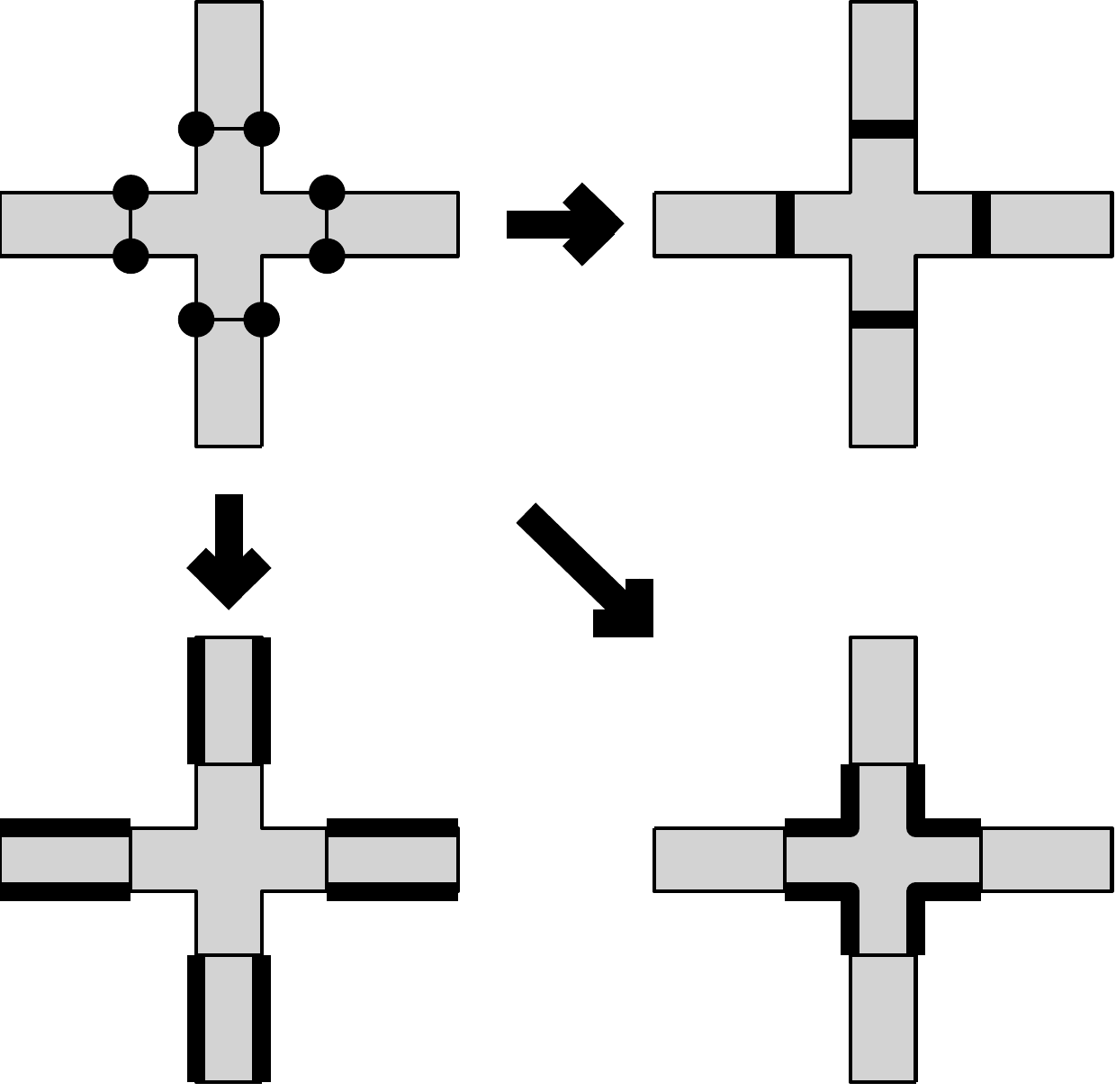}
\caption{The inclusion induced maps $L,M',M''$.}
\label{fig:MapsLMversion1}
\end{figure}

\subsubsection{Constructing the map $\tau$}
\label{subsub:tau}
First, we describe the space $\ker L \cap \ker M$ by making use of the following portion of the Mayer-Vietoris diagram \eqref{eq:bigdiagram}:
	\[\xymatrix@R0.3cm
	{
		&&
		H_0((B \cap Z_1) \cap (B \cap Z_2); \Lambda_S)\ar[r]^{\ \ \ \ \ \ \ \  L}\ar[dd]^{M=M'\oplus M''}  & H_0(Z_1 \cap Z_2; \Lambda_S)&\\
		&&\mbox{\phantom{a}}\\
		&& 
		H_0(B \cap Z_1;\Lambda_S) \oplus H_0 (B \cap Z_2; \Lambda_S).  \\}
	\]
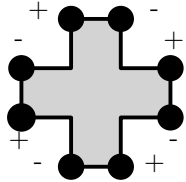
\begin{figure}[b]
\labellist\small\hair 2.5pt
\pinlabel {+} at 150 780
\pinlabel {-} at 300 780
\pinlabel {-} at 150 580
\pinlabel {+} at 300 580
\pinlabel {-} at 125 740
\pinlabel {+} at 125 610
\pinlabel {+} at 325 740
\pinlabel {-} at 325 610
\endlabellist
\centering
\includegraphics[width=0.25\textwidth,scale=0.6]{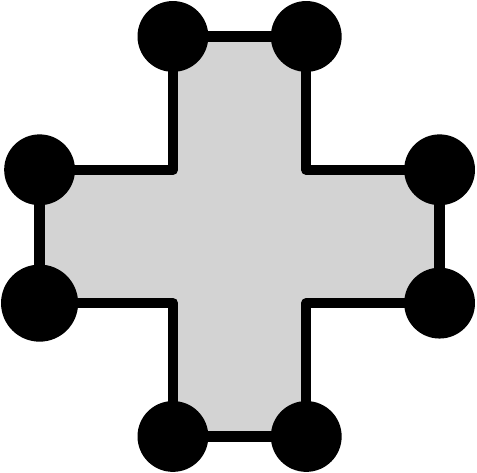}
\caption{The elements $v_{ij}^k$ that span $\ker(L) \cap \ker(M'')$}
\label{fig:Plusminus}
\end{figure}
The spaces $B\cap Z_2$, $Z_1\cap Z_2$ and $B\cap Z_1\cap Z_2$ are all of the form $\bigsqcup_{i<j}X_{ij}\times P$ for some contractible subset $P$ of $\partial K$. Consequently the restriction of the cover to each of these subspaces is the trivial $\mathbb{Z}^\mu$-covering, and the inclusion induced maps $L,M'$ and $M''$ can be understood using Figure \ref{fig:MapsLMversion1}. It follows  that $\ker L \cap \ker M''$ is generated by $c$ linearly independent elements $v_{ij}^k$, one for each clasp $C_{ij}^k$ (as illustrated in Figure \ref{fig:Plusminus}):
\begin{equation}\label{span}
	\ker L \cap \ker M'' = \Span _{\L_S} \{v_{ij}^k \, \vert \, i  < j , \, 1\leq k\leq c(i,j) \}.
	\end{equation}
Since $\ker L\cap \ker M$ is the subspace of $\ker L \cap \ker M''$ which is annihilated by $M'$, we now compute the image of $v_{ij}^k$ under $M'$. Observing that $H_0(B \cap Z_1; \Lambda_S)\cong \bigoplus_{i=1}^\mu H_0(\overline{Y_i\times \{-1\}} \sqcup\overline{Y_i\times \{1\}})\otimes \Lambda_S$ and denoting by $y_i^\pm$ a positive generator of $H_0(\overline{Y_i\times \{\pm 1\}})$, a short computation using Figure \ref{fig:MapsLMversion1} shows that 
 \begin{equation}
 \label{eq:mprime}
 M'v_{ij}^k=t_i(t_j-1) y_i^+ -(t_j-1)y_i^-  - t_j(t_i-1)y_j^+ +(t_i-1)y_j^-.
 \end{equation}
Let $T \colon	C_1(\Gamma) \to  \Span_{\Lambda_S}\{v_{ij}^k\}_{ijk}$ be the $\Z$-linear map defined on generators by
	\begin{equation*}
	T(E_{ij}^k) := \prod_{l\ne i,j} \!(t_l-1)\:v_{ij}^k.
	\end{equation*}
Note that since $\Gamma$ is a graph, its first homology group $H_1(\Gamma)$ is a subspace of $C_1(\Gamma)$.
\begin{proposition}
The restriction of $T$ to $H_1(\Gamma)$ takes values in $\ker L \cap \ker M$.
\end{proposition} 
\begin{proof}
Since this proof consists in a straightforward computation which combines \eqref{eq:mprime} and the definition of $T$, we shall only give the details in a simple case. The general case follows the exact same steps.  Suppose that in our C-complex there are clasps $C_{ij}^1$, $C_{il}^1$, $C_{jk}^1$, with $i<j<k$, then the element $\gamma:=E_{ij}^1+E_{jk}^1-E_{ik}^1$ is in $H_1(\Gamma)$. Since $$ T\gamma = \prod_{l\ne i,j,k}\!(t_l-1)\left((t_k-1)v_{ij}^1+(t_i-1)v_{jk}^1-(t_j-1)v_{ik}^1\right)$$
is contained in $\ker L\cap \ker M''$, it only remains to check that $M'T\gamma=0$. The immediate computation
\begin{equation*} \begin{split}
M'T\gamma = \prod_{l\ne i,j,k}\!(t_l-1)\cdot[&  (t_k-1)(t_i(t_j-1) y_i^+ -(t_j-1)y_i^-   - t_j(t_i-1)y_j^+  +(t_i-1)y_j^-) +\\
&(t_i-1)( t_j(t_k-1) y_j^+ -(t_k-1)y_j^-     - t_k(t_j-1)y_k^+   +(t_j-1)y_k^-)\\
-&(t_j-1)(t_i(t_k-1)y_i^+ -(t_k-1)y_i^-     - t_k(t_i-1)y_k^+   +(t_i-1)y_k^-)]=0,
\end{split}\end{equation*}
which relies on \eqref{eq:mprime} concludes the proof.
\end{proof}
We can now define the $\Lambda_S$-linear map $$\tau\colon H_1(\Gamma)\otimes \Lambda_S\to \ker L \cap \ker M$$ as the extension to $\Lambda_S$-scalars of the $\Z$-linear map $T\colon H_1(\Gamma)\to\ker L \cap \ker M$.

\begin{proposition}
The map $\tau$ is an isomorphism of $\Lambda_S$-modules.
\end{proposition}

\begin{proof} 
As earlier in the paper, let $c$ be the total number of clasps of the C-complex $S$. We will identify the codomain $\ker L \cap \ker M$ with the kernel of a $\mu \times c$ matrix with coefficients in $\Lambda_S$. During this proof, rows of $(\mu \times c)$-matrices will be indexed by integers $l =1,\dots,\mu$ and columns by triples $(i,j,k)$ with $1\leq i<j \leq \mu$ and $1 \leq k \leq c(i,j)$. 
We start by noticing that, thanks to \eqref{span} and \eqref{eq:mprime}, we have
	\begin{equation*}
	\ker L\cap \ker M\,\,=\,\,\bigg\{\smsum{i,j,k}{} a_{ij}^kv_{ij}^k \: \vert \:  \forall \, l=1,\dotsc , \mu, \: \smsum{j,k}{}a_{lj}^k(t_j-1) -  \smsum{i,k}{}a_{il}^k(t_i-1)=0 \bigg\}. 
	\end{equation*}
It follows that $\ker L\cap \ker M$ is the subspace of $\Span_{\Lambda_S}\{v_{ij}^k\}_{ijk}$ given by the null space of the $\mu\times c $ matrix $R$ with $(l,ijk)$-coefficient 
$$(R_l^{ijk})=((t_j-1)\delta_{il}-(t_i-1)\delta_{jl}),$$ 
where $\delta_{ij}$ is the Kronecker delta function.
For each $l$, multiplying the $l$-th row of $R$ by $(t_l-1)$ yields a $(\mu \times c)$-matrix $Q$ whose kernel is still $\ker L\cap \ker M$ and whose $(l,ijk)$-coefficient is
$$(Q_l^{ijk})=((t_i-1)(t_j-1)(\delta_{il}-\delta_{jl})).$$
Next, multiplying each $ijk$-column of $Q$ by $(t_i-1)^{-1}(t_j-1)^{-1}$ results in a matrix $P$ whose $(l,ijk)$-coefficient is $$(P_l^{ijk})=(\delta_{il}-\delta_{jl}).$$
Since $P$ represents the boundary operator $\partial \colon C_1(\Gamma) \to C_0(\Gamma)$,  its kernel over $\Lambda_S$ is isomorphic to $H_1(\Gamma)\otimes \Lambda_S$. Consequently, in order to conclude the proof, it only remains to show that $\ker Q \cong \ker P$. The operations we performed on the columns of $Q$ give rise to a $\Lambda_S$-module isomorphism
$$
v_{ij}^k \mapsto (t_i-1)^{-1}(t_j-1)^{-1}v_{ij}^k\,\,=\,\,\bigg(\smprod{l=1}{\mu} (t_l-1)\bigg)^{-1}\!\!\smprod{l\ne i, j}{}(t_l-1) v_{ij}^k \,\,=\,\, \bigg(\smprod{l=1}{\mu} (t_l-1)\bigg)^{-1} \tau(E_{ij}^k)$$
which restricts to an isomorphism $\ker P \to \ker Q$. Since the $(t_i-1)$ are invertible in $\Lambda_S$, $\tau$ is an isomorphism.
\end{proof}

In order to conclude the proof of the reduction discussed in Subsection \ref{subsub:reduc}, it only remains to prove the next proposition.

\begin{proposition}
\label{prop:bottomsquare}
The bottom square of (\ref{eq:Comut}) commutes.
\end{proposition} 

\begin{proof}
Given a nice cycle $\alpha$ in $S$ and a clasp $C_{ij}^k$, define
$$n_{ij}^k=\begin{cases}
1         & \mbox{if }\alpha \mbox{ crosses } C_{ij}^k  \mbox{ from } i \mbox{ to } j,  \\ 
-1         & \mbox{if }\alpha \mbox{ crosses } C_{ij}^k  \mbox{ from } j \mbox{ to } i,  \\ 
0         & \mbox{if }\alpha \mbox{ does not cross } C_{ij}^k \text{ at all}.
 \end{cases} $$
As we will see, the image of $[\alpha]\otimes 1$ under both $\tau \pi$ and $b \partial \Phi$ will only depend on this combinatorial data. We start with the computation of $\tau \pi([\alpha]\otimes 1)$. From the definitions of $\pi$ and $n_{ij}^k$, it is clear that
$$\pi ([\alpha]\otimes 1)= \sum_{i<j,\: k} n_{ij}^k E_{ij}^k,$$
and consequently the definition of $\tau$ yields
$$\tau \pi ([\alpha]\otimes 1)=\sum_{i<j,\: k} \prod _{l\ne i,j}\! (t_l-1) \:n_{ij}^k v_{ij}^{k}.$$
To conclude, we compute $b \partial \Phi([\alpha]\otimes 1)$. 
Thanks to Proposition \ref{prop:intersects}, we have $
\partial \Phi ([\alpha]\otimes 1)=\sum_\varepsilon \sgn(\varepsilon)t^{\frac{\pmb{1}+\varepsilon}{2}}[\overline{\alpha}^\varepsilon].$
Since the map $b$ is the boundary homomorphism in the Mayer--Vietoris sequence for $B\cap Z=(B \cap Z_1) \cup (B \cap Z_2)$, it only depends on the behavior of $\alpha$ at the clasps. 
For $\varepsilon\in\{\pm 1\}^\mu$, the part of $\overline{\alpha}^\varepsilon$ contained in $B\cap Z_2$ will be either empty or a path connecting $S_i\times \{\varepsilon_i\}$ to $S_j\times \{\varepsilon_j\}$, whose direction depends on the sign of $n_{ij}^k$. Since this data does not depend on the coordinates of $\varepsilon$ different from $i$ and $j$, we will denote such a strand by $\overline{\alpha}_{ij}^{k,\varepsilon_i\varepsilon_j}$. The part of $\partial \Phi ([\alpha]\otimes 1)$ which is contained in $B\cap Z_2$ is hence
$$\smsum{\varepsilon}{} \sgn(\varepsilon)t^{\frac{\pmb{1}+\varepsilon}{2}}\smsum{i<j,\, k}{}\overline{\alpha}_{ij}^{k,\varepsilon_i\varepsilon_j}=
\smsum{i<j,\, k}{} \smsum{\varepsilon'}{} \,\sgn(\varepsilon')\smprod{l\ne i,j}{}\!t^{\frac{1+\varepsilon_l}{2}}\:
  \Big(t_it_j\overline{\alpha}_{ij}^{k,++} -t_i \overline{\alpha}_{ij}^{k,+-} -t_j\overline{\alpha}_{ij}^{k,-+}+ \overline{\alpha}_{ij}^{k,--}\Big),$$ 
  where $\varepsilon'$ now varies in $\{\pm 1 \}^{\mu-2}$.
  We observe now that the boundary of $(t_it_j\overline{\alpha}_{ij}^{k,++} -t_i \overline{\alpha}_{ij}^{k,+-} -t_j\overline{\alpha}_{ij}^{k,-+}+ \overline{\alpha}_{ij}^{k,--})$ is given by $-n_{ij}^kv_{ij}^k$ (see Figure \ref{fig:Curve}). So, taking into account a minus sign coming from the Mayer-Vietoris sequence, we get
$$b\partial \Phi ([\alpha]\otimes 1)= \sum_{i<j,\, k} \sum_{\varepsilon'}\,\sgn(\varepsilon')\prod_{l\ne i,j}\!\! t^{\frac{1+\varepsilon_l}{2}}\: n_{ij}^kv_{ij}^k.$$
As a last step, we apply \eqref{eq:epsilonformula} to the sum over the $\varepsilon'$, obtaining
$$b\partial \Phi ([\alpha]\otimes 1)=\sum_{i<j,\: k}  \prod _{l\ne i,j}\! (t_l-1) \: n_{ij}^k v_{i,j}^{k}.$$
This concludes the proof of the proposition.

\end{proof}
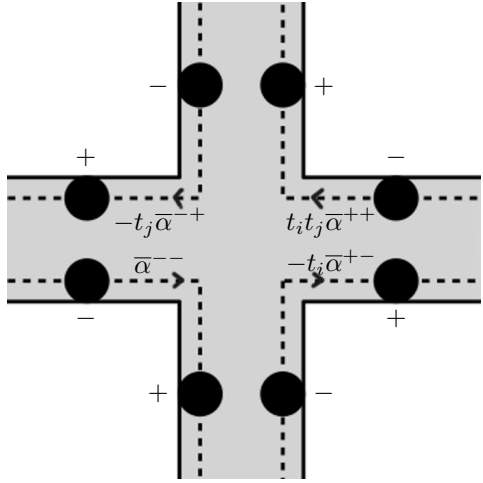
\begin{figure}[t]
\labellist\small\hair 2.5pt
\pinlabel {$-$} at 82 170
\pinlabel {$+$} at 82 341
\pinlabel {$-$} at  158 417
\pinlabel {$+$} at 332 417
\pinlabel {$-$} at   408     341
\pinlabel {$+$} at   408 170
\pinlabel {$-$} at   332 91 
\pinlabel {$+$} at 158 91
\pinlabel {$\overline{\alpha}^{--}$} at 160 230
\pinlabel {$-t_j \overline{\alpha}^{-+}$} at 160 272
\pinlabel {$-t_i\overline{\alpha}^{+-}$} at 340 230
\pinlabel {$t_it_j\overline{\alpha}^{++}$} at 340 272
\endlabellist
\centering
\includegraphics[width=0.4\textwidth,scale=1]{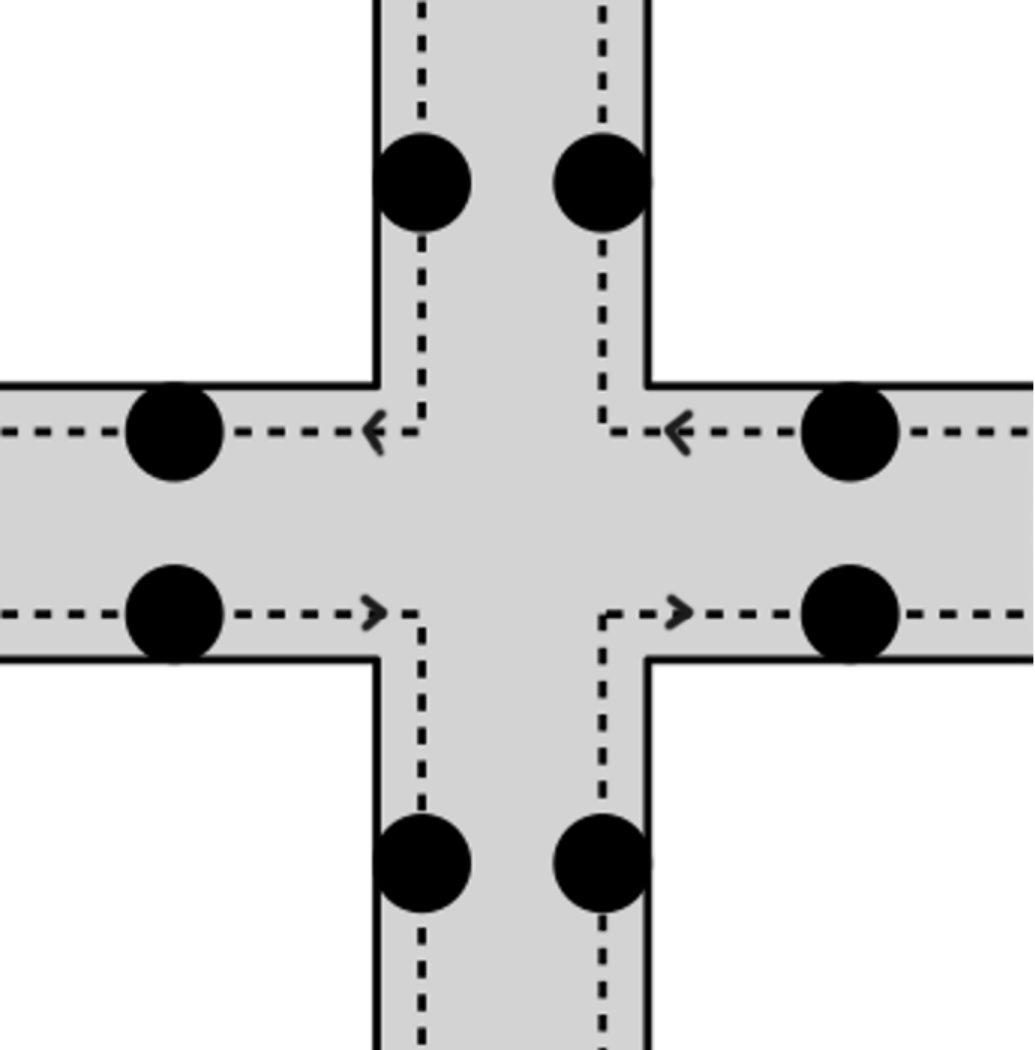}
\caption{A representation of the curve  $(t_it_j\overline{\alpha}_{ij}^{k,++} -t_i \overline{\alpha}_{ij}^{k,+-} -t_j\overline{\alpha}_{ij}^{k,-+}+ \overline{\alpha}_{ij}^{k,--})$, in the case where $n_{ij}^k=1$. Observe that its boundary gives the opposite of the element $v_{ij}^k$ depicted in Figure \ref{fig:Plusminus}.} 
\label{fig:Curve}
\end{figure}
Combining the reduction of Subsection \ref{subsub:reduc} with the results of Subsection \ref{subsub:sigma}, we have now completed the proof of Proposition \ref{prop:phi}. Indeed, since we now know that the maps $\partial, \sigma$ and $\tau$ are all isomorphisms, applying the $5$-lemma to the diagram in Equation (\ref{eq:Comut}) implies that $\Phi \colon H_1(S) \otimes \Lambda_S \to H_2(W_F;\Lambda_S)$ is an isomorphism, as desired.

\subsection{The twisted intersection pairing of $W_F$}
\label{sub:intform}
In Subsection \ref{sub:complement}, we decomposed the exterior of the push-in as $W_F=B \cup Z$, where $B$ was homeomorphic to the 4-ball. Then, in subsection \ref{sub:H2}, we used this decomposition to build an isomorphism $\Phi \colon H_1(S) \otimes \Lambda_S \to H_2(W_F;\Lambda_S)$. In this subsection, we use the resulting basis of $H_2(W_F;\Lambda_S)$ (recall Proposition \ref{prop:phi})  to compute the twisted intersection form $\lambda$ of $W_F$. More precisely, given $\alpha, \beta \in H_1(S)$, we relate $\lambda([\Phi_\alpha],[\Phi_\beta])$ to the matrix $H(t)$ described in Subsection \ref{sub:Seifert}.
\medbreak
Recall from Subsection \ref{sub:intersection forms}, that the formula for $\lambda([\Phi_\alpha], [\Phi_\beta])$ involves the algebraic intersection of $\Phi_\alpha$ with $t^g \Phi_\beta$ for each $g \in \Z^\mu$. In order to pinpoint where these intersections take place, we consider the space 
$$T\,\,:=\,\,\smcup{i=1}{\mu} S_i \times [-1,1] \subseteq S^3$$ 
so that $\Phi _\alpha \subseteq \bigsqcup_{\varepsilon\in \{\pm 1\}^{\mu}}t^{\frac{\pmb{1}+\varepsilon}{2}}\overline{B}\cup p^{-1}(T\star 0),$ where $p \colon \overline{W}_F \to W_F$ denotes the covering  corresponding to the kernel of the abelianization map $\pi_1(W_F) \rightarrow H_1(W_F)$ and $B\cap S^3 = S^3\backslash \mathring{T}$. Moreover,  smoothing the corners, $T$ becomes a smooth submanifold of $S^3$ and, as such, its oriented boundary admits a neighborhood $\partial T \times [-\delta, \delta]$, where the positive part lives outside of $T$.  

The next lemma (whose proof can be understood by looking at Figure \ref{fig:s3}) will make it possible to transfer information from $\partial B$ to the standard $S^3$.

\begin{lemma}
\label{lem:S3}
There exists an orientation preserving homeomorphism between $\partial B$ and $S^3$, which brings $\partial T \star [0,\frac{1}{2}]\subseteq B$ to $\partial T \times [0,\delta]\subseteq S^3$. $($Notice that $\partial T\star \{\frac{1}{2}\}$ is sent to $\partial T \times\{0\}$, while $\partial T\star \{0\}$ is sent to $\partial T \times\{\delta\}$$)$.
\end{lemma}
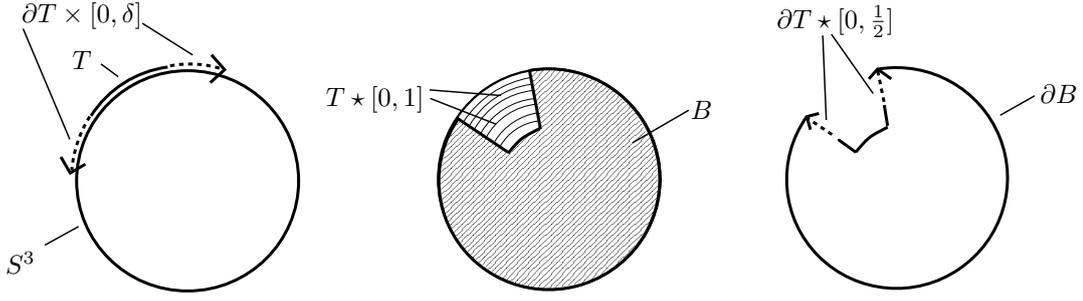
\begin{figure}[h]
\labellist\small\hair 2.5pt
\pinlabel {$\partial T \times [0,\delta]$} at 35 100
\pinlabel {$T$} at 35 84
\pinlabel {$S^3$} at 13 11
\pinlabel {$T \star [0,1]$} at 140 70
\pinlabel {$B$} at 257 66
\pinlabel {$\partial T \star [0,\frac{1}{2}]$} at 305 98
\pinlabel {$\partial B$} at 385 72
\endlabellist
\centering
\includegraphics[width=0.9\textwidth]{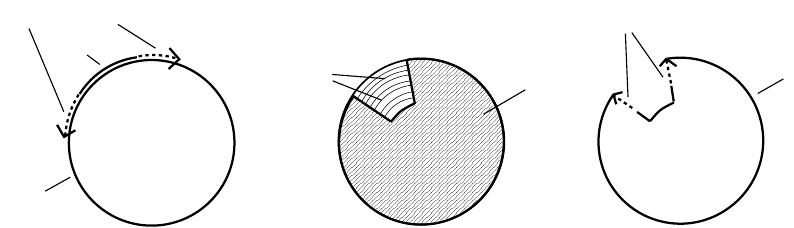}
\caption{Schematic picture of the statement of Lemma~\ref{lem:S3}.}
\label{fig:s3}
\end{figure}
We are now ready to prove the main result of this section.
This result is precisely Theorem~\ref{thm:intform-intro} from the introduction.

\begin{theorem}
\label{thm:intform}
Let $H(t)$ be the matrix described in Subsection \ref{sub:Seifert}, corresponding to the basis $\mathcal{B}$ of $H_1(S)$. With respect to the image of $\mathcal{B}$ by the isomorphism $\Phi$ of Proposition \ref{prop:phi}, the twisted intersection form $\lambda\colon H_2(W_F;\Lambda_S)\times H_2(W_F;\Lambda_S)\to \Lambda_S$ is represented by $H(t)$. More explicitly, if $\Phi_\alpha$ and $\Phi_\beta$ are two of the surfaces constructed above, then we have the formula
\[\lambda([\Phi_\alpha],[\Phi_\beta])=\sum_{\varepsilon} \prod_{i=1}^{\mu} (1-t_i^{\varepsilon_i})\lk(\alpha^{\varepsilon}, \beta).\]
\end{theorem}

\begin{proof}
Fix 
$[\alpha], [\beta] \in \mathcal{B}$, such that $\alpha$ and $\beta$ are the nice representatives in $\mathbf{B}$ we used to define the map $\Phi$. 
We perform homotopies which push $\Phi _\alpha$ and $\Phi _\beta$ inside the interior of $W_F$ in such a way that
$$\displaystyle \Phi _\alpha \subseteq p^{-1}(B)\cup p^{-1}\big(T\star \tmfrac{1}{4}\big)\:\:\:, \:\:\:\displaystyle\Phi _\beta \subseteq p^{-1}(B)\cup p^{-1}\big(T\star \tmfrac{1}{2}\big),$$
and that $\Phi _\alpha$ intersect $t^g\phi_\beta$ transversally for all $g\in \mathbb{Z}^\mu$.
Recalling the construction of $\Phi_\alpha$ and $\Phi_\beta$, we see that the algebraic intersection of $\Phi_\alpha$ with $t^g \Phi_\beta$ can now only happen in the disjoint union of the $4$-balls $t^h\overline{B}$.  
It follows that  
$$ \lambda([\Phi_\alpha],[\Phi_\beta]) \overset{\text{def}}{=}\sum_{g\in \mathbb{Z}^{\mu}} s(\Phi_\alpha, t^g \, \Phi_\beta) \, t^g   =\sum_{g\in \mathbb{Z}^{\mu}} \sum_{\varepsilon, \varepsilon' \in \{\pm 1\}^{\mu}} \sgn(\varepsilon)\sgn(\varepsilon') s(t^{\frac{\pmb{1}+\varepsilon}{2}}\overline{S}_{\alpha}^{\varepsilon}, t^gt^{\frac{\pmb{1}+\varepsilon'}{2}}\overline{S}_{\beta}^{\varepsilon'})t^g.$$
Moreover, the two surfaces $t^{\frac{\pmb{1}+\varepsilon}{2}}S_{\alpha}^{\varepsilon}$ and $t^g t^{\frac{\pmb{1}+\varepsilon}{2}}S_{\beta}^{\varepsilon'}$ only intersect when they belong to the same lift of $B$, i.e. when $t^gt^{\frac{\pmb{1}+\varepsilon'}{2}}=t^{\frac{\pmb{1}+\varepsilon}{2}}$. This occurs precisely when $g=\frac{\varepsilon-\varepsilon '}{2}$ and, in this case, translation invariance gives us
 $$ s(t^{\frac{\pmb{1}+\varepsilon}{2}}\overline{S}_{\alpha}^{\varepsilon}, t^gt^{\frac{\pmb{1}+\varepsilon'}{2}}\overline{S}_{\beta}^{\varepsilon'})=s(S_\alpha^\varepsilon,S_\beta^{\varepsilon'}).$$
Homotope $\overline{S}_{\alpha}^{\varepsilon}, \overline{S}_{\beta}^{\varepsilon'}\subseteq \overline{B}$ so that their boundaries are respectively $\overline{\alpha}^{\varepsilon}\star \{\frac{1}{4}\}$ and $ \overline{\beta}^{\varepsilon'}\star \{\frac{1}{2}\}$. Consequently, the algebraic intersection $s(S_\alpha^\varepsilon,S_\beta^{\varepsilon'})$ coincides with the linking number in $\partial  \overline{B}$ of $\overline {\alpha}^{\varepsilon}\star \{\frac{1}{4}\}$ and $\overline{\beta}^{\varepsilon'}\star \{\frac{1}{2}\}$, which in turn equals the linking number in $\partial B$ of $\alpha^{\varepsilon}\star \{\frac{1}{4}\} $ and $\beta^{\varepsilon'}\star \{\frac{1}{2}\}$. Lemma \ref{lem:S3} now provides the existence of an orientation preserving homeomorphism from $\partial B$ to $S^3$ that brings $\alpha^{\varepsilon}\star \{\frac{1}{4}\} $ to $\alpha ^{\varepsilon} \times \{\frac{\delta}{2}\}$, and $\beta^{\varepsilon'}\star \{\frac{1}{2}\}$ to $\beta ^{\varepsilon'}=\beta^{\varepsilon'}\times \{0\}$. As a consequence, we obtain
	$$s(S_{\alpha}^{\varepsilon}, S_{\beta}^{\varepsilon'})= \lk(\alpha ^{\varepsilon} \times \{\delta/ 2\},  \beta ^{\varepsilon'}) = \lk (\alpha ^{\varepsilon}, \beta ).$$
Putting everything together, we get
$$ 	\lambda([\Phi_\alpha],[\Phi_\beta]) 
=\sum_{\varepsilon, \varepsilon' \in \{\pm 1\}^{\mu}} \sgn(\varepsilon)\sgn(\varepsilon') \lk(\alpha^{\varepsilon}, \beta) t^{\frac{\varepsilon - \varepsilon'}{2}}.$$
We will now algebraically manipulate this last expression in order to get the desired formula. Factoring out the terms involving $\varepsilon'$ and using \eqref{eq:epsilonformula} (applied to the variables $t_i^{-1}$), we rewrite this as
$$\lambda([\Phi_\alpha],[\Phi_\beta])
= \sum_{\varepsilon}\sgn (\varepsilon)\lk(\alpha^{\varepsilon}, \beta)\, t^{\frac{\varepsilon + \pmb{1}}{2}} \sum_{\varepsilon'}\sgn(\varepsilon')\, t^{-\frac{\pmb{1}+\varepsilon'}{2}}
=\sum_{\varepsilon}\sgn (\varepsilon)\lk(\alpha^{\varepsilon}, \beta)\, t^{\frac{\varepsilon + \pmb{1}}{2}}\prod_{i=1}^{\mu}(t_i^{-1}-1).$$
Then, from the identity $\varepsilon_i \, t_i^{\frac{\varepsilon_i + 1}{2}} (t_i^{-1}-1)=1-t_i^{\varepsilon_i}$, we get
$$\lambda([\Phi_\alpha],[\Phi_\beta])=\sum_{\varepsilon} \prod_{i=1}^{\mu} (1-t_i^{\varepsilon_i})\lk(\alpha^{\varepsilon}, \beta),$$
which concludes the proof of the theorem.
\end{proof}

\section{Proof of Theorem \ref{thm:main}}
\label{sec:proof}

\subsection{Preliminary lemmas}
\label{sub:lemmas}
Let $L$ be a colored link and let $S$ be a C-complex for $L$. 
Given the exterior $W_F$ of the push-in $F$ of the C-complex $S$, we saw in Section \ref{sec:intform} that the intersection form on $H_2(W_F;\Lambda_S)$ could be computed in terms of generalized Seifert matrices. In this subsection, we prove some technical lemmas which will enable us to relate the above mentioned intersection form to the Blanchfield pairing on $\partial W_F$.

\begin{lemma}
\label{lem:ev}
If the C-complex $S$ is totally connected, then the evaluation map 
$$\ev \circ \kappa \colon H^2(W_F;\Lambda_S) \rightarrow \overline{\Hom_{\Lambda_S}(H_2(W_F;\Lambda_S),\Lambda_S)}$$ is an isomorphism.
\end{lemma}

\begin{proof}
Since the map 
\[ \kappa\colon H^2(W_F;\Lambda_S) \,\,\rightarrow\,\, H_2\Big(\overline{\Hom_{\Lambda_S}\big(C_*(W_F;\Lambda_S),\Lambda_S\big)}\Big)\]  is easily seen to be an isomorphism. It suffices to show that $\ev$ is an isomorphism. Now we consider the universal coefficient spectral sequence, see e.g.\ \cite[Theorem 2.3]{Lev}, which starts at~$E^2_{p,q}= \text{Ext}^q_{\Lambda_S}(H_p(X;\Lambda_S),\Lambda_S)$, has differential of degree~$(1-r,r)$ and converges to~$H_*\big({\Hom_{\Lambda_S}\big(\overline{C_*(W_F;\Lambda_S)},\Lambda_S\big)}\big)$. We have~$H_0(W_F;\Lambda_S)=H_1(W_F;\Lambda_S)=0$ by Lemma~\ref{lem:lambdaS} and Corollary~\ref{cor:H10}. Therefore 
\[ \begin{array}{rcl} \text{ev}\colon  H_2\big({\Hom_{\Lambda_S}(\overline{C_*(W_F;\Lambda_S)},\Lambda_S)}\big) &\rightarrow&
\text{Ext}^0_{\Lambda_S}(H_2(W_F;\Lambda_S),\Lambda_S)\\
&=& {\Hom_{\Lambda_S}\big(H_2(C_*(W_F;\Lambda_S)),\Lambda_S\big)}\end{array}\]
is an isomorphism.
\end{proof}

\begin{lemma}
\label{lem:W_FFtorsion}
$H_1(\partial W_F; \Lambda_S)$ is isomorphic to~$H_1(X_L;\Lambda_S)$. 
\end{lemma}

\begin{proof}
The boundary of~$W_F$ decomposes as~$X_L \cup_{L \times S^1} \overline{\partial \nu F}$ and the resulting Mayer-Vietoris sequence with~$\Lambda_S$-coefficients is given by 
$$ H_1(L \times S^1;\Lambda_S) \rightarrow H_1(X_L;\Lambda_S) \oplus H_1(\overline{\partial \nu F};\Lambda_S) \rightarrow H_1(\partial W_F;\Lambda_S) \rightarrow H_0(L \times S^1;\Lambda_S).~$$
As the restriction of~$H_1(W_F) \rightarrow \mathbb{Z}^\mu $ to~$H_1(L_i \times S^1;\mathbb{Z})$ sends each meridian to~$t_i$, Lemma \ref{lem:lambdaS} ensures that~$C_*(L \times S^1;\Lambda_S)$ is acyclic. Consequently, it only remains to show that~$H_1(\overline{\partial \nu F};\Lambda_S)$ vanishes. Away from the $c$ double points of $F$, the boundary of a tubular neighborhood of each~$F_i$ consists in~$F_i \times S^1$ for $i=1,\dots,\mu$. Given a double point $C_{ij}^k \in F_i \cap F_j$, remove the open disk $D_{ij}^k$ which consists in the component of $\nu (F_j\cap F_i)$ containing $C_{ij}^k$. Repeating the process for each double point produces punctured surfaces~$X_1,\dots,X_\mu$. The manifold~$\overline{\partial \nu F}$ can now be recovered from the union of the~$X_i \times S^1$ by gluing each $D_{ij}^k \times S^1$ along the tori~$\partial D_{ij}^k \times S^1$.
The corresponding Mayer-Vietoris exact sequence yields
$$  \smoplus{i<j}{} \smoplus{k=1}{c(i,j)} H_1\big( \partial D_{ij}^k \times S^1;\Lambda_S \big) \rightarrow \smoplus{i=1}{\mu} H_1(X_i \times S^1;\Lambda_S ) \rightarrow H_1(\overline{\partial \nu F};\Lambda_S) \rightarrow \smoplus{i<j}{} \smoplus{k=1}{c(i,j)} H_0 \big( \partial D_{ij}^k \times S^1;\Lambda_S \big).$$
As each~$S^1$ factor arises as a meridian of the link~$L$, one can apply Lemma \ref{lem:lambdaS} and the claim immediately follows.
\end{proof}

\begin{lemma}
\label{lem:cohomoIsom}
If $H_1(X_L;\Lambda_S)$ is torsion, then the inclusion induced map $ H^2(W_F,\partial W_F;Q) \to H^2(W_F;Q)$ is an isomorphism.
\end{lemma}

\begin{proof}
As the Alexander module of~$L$ is torsion over~$\Lambda_S$, and~$Q$ is flat over~$\Lambda_S$, Lemma \ref{lem:W_FFtorsion} implies that~$H_1(\partial W_F;Q)$ vanishes. By an Euler characteristic argument, we see that $H_2(\partial W_F;Q)=0$. Thus the long exact sequence of the pair~$(W_F,\partial W_F)$ with $Q$-coefficients implies that $H_2(W_F;Q) \rightarrow H_2(W_F,\partial W_F;Q)$ is an isomorphism and the result follows by duality.
\end{proof}

\subsection{Conclusion of the proof}
\label{sub:conclusion}
Throughout this section, let $S$ be a totally connected C-complex for a $\mu$-colored link $L$. We denote by $W=W_F$ the exterior of the push-in $F$ of the C-complex $S$ as defined in Section~\ref{sub:complement}. Corollary \ref{cor:H10} together with the long exact sequence of the pair~$(W,\partial W)$ yield the short exact sequence
\begin{equation}
H_2(W;\Lambda_S) \rightarrow H_2(W,\partial W;\Lambda_S) \xrightarrow{\partial} H_1(\partial W;\Lambda_S) \rightarrow 0.
\end{equation}
Combining this with the isomorphism provided by Lemma \ref{lem:cohomoIsom} leads to the following diagram
\begin{equation}
\label{eq:LargeDiagram}
\xymatrix@R0.6cm@C-1cm{
H_2(W;\Lambda_S) \ar[rr] \ar[dd]^{\Theta} && H_2(W,\partial W;\Lambda_S) \ar[rrrrr]^-\partial \ar[d]^{PD} &&&&& H_1(\partial W;\Lambda_S)\ar[d]^{PD}\ar[rr] \ar@/^/[dddrr]^{\Omega} &&0\\
&& H^2(W;\Lambda_S) \ar[rrrrr]\ar[d]\ar[lld]^{\ev \circ \kappa}&&&&& H^2(\partial W;\Lambda_S)\ar[d]^{BS^{-1}} && \\
 \overline{\Hom(H_2(W;\Lambda_S),\Lambda_S)} \ar[d] && H^2(W;Q)\ar[lld]^{\ev \circ \kappa}\ar[d]^{\cong} &&&&& H^1(\partial W;Q/\Lambda_S)\ar[rrd]^{\ev \circ \kappa} \ar[d]&&   \\ 
 \overline{ \Hom(H_2(W;\Lambda_S),Q) } \ar[d]^{}&& H^2(W,\partial W;Q)\ar[rrrrr]\ar[lld]^{\ev \circ \kappa}&&&&& H^2(W,\partial W;Q/\Lambda_S) \ar[rrd]^{\ev \circ \kappa}&&   \overline{ \Hom(H_1(\partial W;\Lambda_S),Q/\Lambda_S) } \ar[d]_{\partial^{\sharp}} \\
 \overline{\Hom(H_2(W,\partial W;\Lambda_S),Q)} \ar[rrrrrrrr]^{} && &&&& &&&  \overline{ \Hom(H_2(W,\partial W;\Lambda_S),Q/\Lambda_S) },
}
\end{equation}
where all homomorphisms are understood to be homomorphisms of $\Lambda_S$-modules.
The top middle square commutes by duality. All the squares involving evaluation maps clearly commute, while the upper left (resp. right) square commutes by definition of the intersection (resp. Blanchfield) pairing. Finally the middle rectangle \textit{anti}-commutes thanks to the following algebraic lemma whose statement is a variation on \cite[Lemma 4.4]{BLLV}. 

\begin{lemma}
\label{lem:cube}
Consider the following commutative diagram of chain complexes of $\Lambda_S$-modules:
$$ \xymatrix@R0.5cm{
  &0 \ar[d]&  0\ar[d]&  0\ar[d]&  \\
0 \ar[r] &A\ar[d] \ar[r]& B \ar[d]_{v_B}\ar[r]^{h_B}& C \ar[r]\ar[d]& 0 \\
0 \ar[r] &D \ar[r]^{h_D}\ar[d]_{v_D}& E \ar[r]\ar[d]& F \ar[r]\ar[d]& 0 \\
0 \ar[r] &H \ar[r]\ar[d]& J \ar[r]\ar[d]& K \ar[r]\ar[d]& 0  \\
  &0 &  0&  0.& 
}$$
If $J$ is acyclic, then the following diagram anti-commutes
$$ \xymatrix@R0.5cm{
H_{-2}(D) \ar[r]^{v_D}\ar[d]^{h_D} & H_{-2}(H) \\
H_{-2}(E) \ar[d]^{{v_B}^{-1}}_\cong & H_{-1}(K)\ar[d] \ar[u]_\cong \\
H_{-2}(B) \ar[r]^{h_B}& H_{-2}(C), \\
}
$$
where the right vertical maps are boundary homomorphisms.
\end{lemma}

Consequently, if one defines a pairing $\theta$ on $H_2(W,\partial W;\Lambda_S)$ by the composition
\begin{align*}
H_2(W,\partial W;\Lambda_S) \xrightarrow{\operatorname{PD}} H^2(W;\Lambda_S) \rightarrow H^2(W;Q) &\cong H^2(W,\partial W;Q) \\
 &\rightarrow \overline{\Hom_{\Lambda_S}(H_2(W,\partial W;\Lambda_S),Q)},
\end{align*} 
where for the third map, Lemma \ref{lem:cohomoIsom} was used, then the commutativity of the diagram in equation (\ref{eq:LargeDiagram}) immediately implies the commutativity of 
\begin{equation}
\label{eq:DiagramThreeLines}
\xymatrix@C1.4cm@R0.5cm{ H_2(W;\Lambda_S) \times H_2(W;\Lambda_S) \ar[r]^{\ \ \ \ \  \ \ \ \ \ \ \ \ -\lambda}\ar[d]^{} & \Lambda_S\ar[d]  \\
H_2(W,\partial W;\Lambda_S) \times H_2(W,\partial W;\Lambda_S) \ar[r]^{\ \ \ \ \ \ \ \ \ \ \ \ \ \ \ \ \ \ \ \  -\theta}\ar[d]^{\partial \times \partial } & Q \ar[d]   \\
H_1(\partial W;\Lambda_S) \times H_1(\partial W;\Lambda_S) \ar[r]^{\ \ \ \ \ \ \ \ \ \  \ \ \operatorname{Bl}(L)}& Q/\Lambda_S.
}
\end{equation} 
We pick our basis $\mathcal{B}$ for~$H_1(S)$. This basis yields generalized Seifert matrices~$A^\varepsilon$ for $L$ and it induces a basis~$\mathcal{C}$ of~$H_2(W;\Lambda_S)$ (recall Proposition \ref{prop:phi}). We endow $\overline{\Hom_{\Lambda_S}(H_2(W;\Lambda_S),\Lambda_S)}$ with the corresponding dual basis $\mathcal{C}^*$. Now we consider the following commutative diagram of $\Lambda_S$-homomorphisms
\begin{equation}
\label{eq:DiagramForBasis}
 \xymatrix@R0.5cm@C0.9cm{
H_2(W;\Lambda_S)\ar[d]\ar[r]^-{PD} & H^2(W,\partial W;\Lambda_S)\ar[d]\\
H_2(W,\partial W;\Lambda_S)\ar[r]^-{PD} &H^2(W;\Lambda_S)\ar[r]^-{\ev \circ \kappa}
&\overline{\Hom_{\Lambda_S}(H_2(W;\Lambda_S),\Lambda_S)}.}
\end{equation}
Here the bottom-right map is an isomorphism by  Lemma \ref{lem:ev}.
\begin{claim}
With respect to the bases $\mathcal{C}$ and $\mathcal{C}^*$, the homomorphism $$\Theta: H_2(W;\Lambda_S)\to  \overline{\Hom_{\Lambda_S}(H_2(W;\Lambda_S),\Lambda_S)}$$ is represented by the matrix $H(t)^T$.
\end{claim}

\begin{proof}
By definition, the intersection pairing is given by $\lambda(x,y)= \Theta(y)(x)$.
In Theorem \ref{thm:intform}, we found out that the matrix representing $\lambda$ with respect to the basis $\mathcal{C}$ is $H(t)$. If  $\mathcal{C}= \{x_1,\dotsc , x_n \}$, we can hence write
$$\Theta(x_j)\,\,=\,\,
 \sum_{i=1}^n \overline{\Theta(x_j)(x_i)} \, x_i^*\, \,=\,\, \sum_{i=1}^n \overline{\lambda(x_i,x_j)} \, x_i^*\,\, =\,\, \sum_{i=1}^n \overline{H(t)_{ij}} \, x_i^*\,\, =\,\, \sum_{i=1}^n H(t)_{ji} \, x_i^* ,$$
which proves the claim. Notice that in the first equality, the bar appears because of the involuted structure of $\overline{\Hom_{\Lambda_S}(H_2(W;\Lambda_S),\Lambda_S)}$, while the last equality is true because $H(t)$ is a hermitian matrix.
\end{proof}
Now we equip $H_2(W,\partial W;\Lambda_S)$ with the basis induced from $\mathcal{C}^*$ and the bottom two isomorphisms in the diagram of equation (\ref{eq:DiagramForBasis}). Then, by commutativity of the diagram, the left vertical map is also represented by $H(t)^T$. Rewriting the commutative diagram  of equation (\ref{eq:DiagramThreeLines}) in terms of these bases, and setting~$n=\text{rk}_{\mathbb{Z}} H_1(S)$ one obtains
\[ \xymatrix@C2.3cm@R0.5cm{\Lambda_S^n \times \Lambda_S^n \ar[r]^{(a,b) \mapsto -a^T H(t) \overline{b}}\ar[d]_{(a,b) \mapsto (H(t)^Ta,H(t)^Tb)} & \Lambda_S \ar[d]  \\
\Lambda_S^n \times \Lambda_S^n\ar[d]\ar[r]^{(a,b) \mapsto -a^T H(t)^{-1} \overline{b}} & Q \ar[d] \\
H_1(\partial W;\Lambda_S) \times H_1(\partial W;\Lambda_S)\ar[r]^{\ \ \ \ \ \ \ \ \ \ \ \ \ \ \operatorname{Bl}(L)} & Q/\Lambda_S.
} \]
Here the middle horizontal map is determined by the top horizontal map, the vertical maps and the commutativity. So the Blanchfield pairing on~$H_1(\partial W; \Lambda_S)$ is isometric to the pairing

\begin{align*}
\Lambda_S^n / H(t)^T \Lambda_S^n \times \Lambda_S ^n / H(t)^T \Lambda_S^n &\rightarrow Q/\Lambda_S \\
(a,b) &\mapsto -a^T H(t)^{-1}\overline{b}.
\end{align*}
Using Lemma \ref{lem:W_FFtorsion}, this is also the Blanchfield pairing on~$H_1(X_L;\Lambda_S)$, concluding the proof.

\subsection{The classical formula for the Blanchfield form}
\label{section:comparison}
 As a reality check we  verify that 
Theorem~\ref{thm:blanchfield-in-terms-of-seifert-matrix-for-knot} is indeed a consequence of Theorem~\ref{thm:main}.

Let $K$ be an oriented knot and let~$A$ be a Seifert matrix for~$K$ of size~$2g$. Following the convention of~\cite{Rolfsen} we have $A^-=A$ and $A^+=A^T$. Therefore
\[ H(t)\,\,=\,\,(1-t)A^T+(1-t^{-1})A\,\,=\,\,-(t^{-1}-1)(A-tA^T).\]
We consider the following commutative diagram (from which Theorem \ref{thm:blanchfield-in-terms-of-seifert-matrix-for-knot} follows): 
\[ 
\xymatrix@R0.4cm@C0.2cm{ H_1(X_K;\mathbb{Z}[t^{\pm 1}])\times H_1(X_K;\mathbb{Z}[t^{\pm 1}])\ar[rrrrr]^-{\op{Bl}(K)}\ar[dd]^\cong 
\ar[drrr]&&&&&&\mathbb{Q}(t)/\mathbb{Z}[t^{\pm 1}]\ar[dd]\\
&&&\Delta_K(t)^{-1}\mathbb{Z}[t^{\pm 1}]/\mathbb{Z}[t^{\pm 1}]
\ar@{^(->}[drrr]\ar[urrr]
 \\
 \Lambda_S^{2g}/H(t)^T\Lambda_S^{2g} \times  \Lambda_S^{2g}/H(t)^T\Lambda_S^{2g} \ar[rrrrr]^-{(a,b) \mapsto -a^T H(t)^{-1}\overline{b}}\ar[dd]^{=} &&&&&&\mathbb{Q}(t)/\Lambda_S\ar[dd]^=\\\\
\Lambda_S^{2g}/(tA-A^T)\Lambda_S^{2g} \times  \Lambda_S^{2g}/(tA-A^T)\Lambda_S^{2g} \ar[rrrrr]^-{(a,b) \mapsto a^T (t^{-1}-1)^{-1}(A-tA^T)^{-1}\overline{b}}\ar[dd]^{(v,w)\mapsto ((t-1)^{-1}v,(t-1)^{-1}w)} &&&&&&\mathbb{Q}(t)/\Lambda_S\ar[dd]^=\\\\
\Lambda_S^{2g}/(tA-A^T)\Lambda_S^{2g} \times  \Lambda_S^{2g}/(tA-A^T)\Lambda_S^{2g} \ar[rrrrr]^-{(a,b) \mapsto a^T (t-1)(A-tA^T)^{-1}\overline{b}} &&&&&&\mathbb{Q}(t)/\Lambda_S.\\
}
\]
Here we use the following:
\begin{enumerate}
\item The Blanchfield pairing of a knot takes values in $\Delta_K(t)^{-1}\mathbb{Z}[t^{\pm 1}]/\mathbb{Z}[t^{\pm 1}]$. Since $\Delta_K(1)=\pm 1$ it is straightforward to see that the map
\[ \Delta_K(t)^{-1}\mathbb{Z}[t^{\pm 1}]/\mathbb{Z}[t^{\pm 1}]\to \Q(t)/\Lambda_S\]
is a monomorphism.
\item The top rectangle commutes by Theorem~\ref{thm:main}.
\item In the middle rectangle, we use that $-H(t)=(t^{-1}-1)(A-tA^T)$, which also implies that $H(t)^T=-(t^{-1}-1)(A^T-tA)$. Moreover, we use that $t^{-1}-1=(1-t)t^{-1}$ is a unit in $\Lambda_S$.
\item A straightforward calculation shows that the last rectangle commutes.
\end{enumerate}

\bibliographystyle{plain}
\bibliography{Biblioblanchfieldseifert}

￼
\end{document}